\documentclass[12pt]{article}
\usepackage[utf8]{inputenc}
\usepackage{amsmath}   
\usepackage{mathtools} 

 \usepackage{amsfonts}
\usepackage{amssymb}

\usepackage{amsthm}

\usepackage{mathrsfs}   
\usepackage{bm}         
\usepackage{braket}     
\usepackage{stmaryrd}   

\usepackage{graphicx}
\usepackage[dvipsnames]{xcolor}
\usepackage{float}
\usepackage{caption}
\usepackage{subcaption}

\usepackage{array}
\usepackage{multicol}
\usepackage{nicematrix}
\usepackage{enumitem}
\usepackage{url}
\usepackage{verbatim}   

\usepackage{tikz}
\usetikzlibrary{patterns}
\usetikzlibrary{shapes.geometric, arrows}
\usetikzlibrary{arrows.meta, chains, positioning, quotes}
\usetikzlibrary{calc}

\newcolumntype{P}[1]{>{\centering\arraybackslash}p{#1}}
\newcolumntype{M}[1]{>{\centering\arraybackslash}m{#1}}

\definecolor{ugentblue}{HTML}{1E64C8}
\definecolor{ugentyellow}{HTML}{FFD200}

\theoremstyle{plain}
\newtheorem{thm}{Theorem}[section]
\newtheorem{lem}[thm]{Lemma}
\newtheorem{prop}[thm]{Proposition}
\newtheorem{cor}[thm]{Corollary}

\theoremstyle{definition}
\newtheorem{defn}{Definition}[section]

\theoremstyle{plain}

\newtheorem{ex}{Example}[section]

\newtheorem{rem}[thm]{Remark}

\newcommand{\niceremark}[3]{}
\newcommand{\niceremarkcolor}[4]{}
\newcommand{\aida}[2][]{}
\newcommand{\nichola}[2][]{}

\usepackage[maxbibnames=9,giveninits=true,doi=false,isbn=false]{biblatex}
\usepackage[colorlinks=true,citecolor=blue,linkcolor=blue,urlcolor=blue]{hyperref}
\usepackage{cleveref}

\title{Spectral and Additive Combinatorial Methods for Cycles and Absorbing Sets in Lifted-Product Quantum LDPC Codes}
\author{Aida Abiad\thanks{\texttt{a.abiad.monge@tue.nl}, Department of Mathematics and Computer Science, Eindhoven University of Technology, The Netherlands; Department of Mathematics and Data Science of Vrije Universiteit Brussel, Belgium} \and Nichola Castriota\thanks{\texttt{Nichola.Castriota@UGent.be}, Department of Mathematics: Analysis, Logic and Discrete Mathematics, Ghent University, Belgium}}
\date{}

\begin{document}

\maketitle

\begin{abstract}\noindent
The finite-length performance of quantum low-density parity-check (LDPC) codes under iterative decoding is governed by small substructures of the Tanner graph, principally short cycles and absorbing sets. While the classical theory of these substructures for quasi-cyclic codes is well developed through discrete Fourier transform (DFT) methods, these tools do not directly address the two-block tensor structure $H_X = [\,\widetilde{H}_1 \mid I \otimes \widetilde{B}^T\,]$ of the lifted-product (quasi-cyclic generalised hypergraph product, QC-GHP) codes that dominate current quantum LDPC constructions.

In this paper we develop a quantum-specific spectral framework that exploits this structure. At its core is a DFT block-diagonalisation of $H_X H_X^T$ that reduces moment-trace and cycle computations from an $(r_1\ell)\times(r_1\ell)$ matrix to a sum of $\ell$ small $r_1\times r_1$ Hermitian matrices, with the second block entering only as a scalar shift. From this result we derive a closed-form $4$-cycle count for generalised bicycle codes via additive energies, a joint Sidon characterisation of girth $6$ in the spirit of Fossorier's classical criterion, a Fourier expression for the number of $(3,3)$ elementary absorbing sets in column-weight-$3$ codes via the Wang--Dolecek--Wesel triangle bijection, and a lower bound on stopping-set sizes using the expander mixing lemma.
\end{abstract}

\section{Introduction}\label{sec:intro}

Error-correcting codes underlie every modern communication and storage system, from cellular networks to deep-space links to the read channels in solid-state drives. In the quantum setting they play an even more fundamental role: physical qubits lose their state on microsecond timescales, so any scalable architecture must encode its information into a quantum error-correcting code able to detect and repair faults faster than they accumulate. \emph{Low-density parity-check (LDPC) codes}, introduced by Gallager~\cite{gallager1962} for the classical setting and later studied in the quantum setting, beginning with the sparse-graph constructions of MacKay, Mitchison, and McFadden~\cite{mackay2004sparse} and notably through the generalised-bicycle construction of Kovalev and Pryadko~\cite{kovalev2013}, are among the leading candidate families for both regimes because they admit fast iterative decoders whose complexity grows only linearly with the block length.

In practice, at the realistic block lengths used in hardware, the iterative decoder of an LDPC code does not meet the asymptotic guarantees promised by coding theory: its failure probability eventually flattens out as the noise decreases, and the configurations responsible for this error floor are small recurring patterns in the bipartite \emph{Tanner graph} associated to the code's parity-check matrix~\cite{richardson2003, dolecek2010}. Two such patterns have repeatedly been singled out. \emph{Short cycles} disrupt the independence assumptions on which iterative decoding rests and cause the decoder to oscillate, and \emph{absorbing sets} are small vertex configurations where erroneous messages reinforce one another and become locally stable; the latter have been recognised as the dominant failure mechanism of iterative decoders across a wide range of channels~\cite{dolecek2010, mcmillon2023}. Morris, Pllaha, and Kelley~\cite{morris2024} have recently shown that absorbing sets also drive decoder failure in Calderbank--Shor--Steane (CSS) quantum LDPC codes, so the same two graphical patterns govern the finite-length performance of classical and quantum LDPC codes alike.

Controlling these patterns turns the engineering question of code design into a combinatorial problem: how many few short cycles can a sparse Tanner graph have, and how small can its absorbing sets be? On the classical side, a considerable amount of work tackles this question through the spectrum of the parity-check matrix. For quasi-cyclic (QC) LDPC codes built from a cyclic lift, the discrete Fourier transform (DFT) block-diagonalises the relevant matrices and reduces cycle counts to small eigenvalue computations: this is the eigenvalue reduction of Smarandache and Flanagan~\cite{SF2009} and the directed-edge-matrix recipe of Karimi and Banihashemi~\cite{karimi2012}, while Fossorier~\cite{fossorier2004} characterised the girth of QC codes built from circulant permutation matrices through exponent-difference conditions and Wang, Dolecek, and Wesel~\cite{wangdolecekwesel2012} gave a bijection between the smallest absorbing sets of column-weight-$3$ codes and triangles in an associated graph. On the quantum side, by contrast, cycle and absorbing-set counts of concrete code families have so far been obtained largely by direct enumeration~\cite{panteleev2021degenerate} rather than through structural spectral tools, and the absorbing-set theory of CSS codes is only beginning to be developed~\cite{morris2024}. A main obstacle is that the leading quantum LDPC constructions, the lifted-product (and quasi-cyclic generalised hypergraph product, QC-GHP) codes of Panteleev and Kalachev~\cite{panteleev2021degenerate}, have an $X$-check matrix of the two-block tensor form $H_X = [\,\widetilde{H}_1 \mid I \otimes \widetilde{B}^T\,]$, which the single-block classical DFT machinery does not directly address.

In this work we address this gap by developing a quantum-specific spectral framework that exploits the two-block tensor form rather than working around it (Theorem~\ref{prop:spectral-trace}). The main tool is a DFT block-diagonalisation of $H_X H_X^T$  that captures the contribution of the second block as a scalar shift, reducing moment-trace and cycle computations on an $(r_1\ell)\times(r_1\ell)$ ambient matrix to a sum of $\ell$ small $r_1\times r_1$ Hermitian matrices. As a consequence of the proposed spectral framework, we obtain, for generalised bicycle (GB) codes, a closed-form count of $4$-cycles in terms of the additive energies of the support sets of the defining polynomials (Proposition~\ref{prop:4cycles}) together with a joint Sidon-type characterisation of girth~$6$ (Corollary~\ref{cor:sidon}) that parallels, in the single-circulant setting, Fossorier's classical exponent-difference criterion~\cite{fossorier2004}; through the triangle bijection of Wang, Dolecek, and Wesel~\cite{wangdolecekwesel2012}, we derive a Fourier expression for the number of $(3,3)$ elementary absorbing sets in column-weight-$3$ QC-GHP codes (Proposition~\ref{prop:33-spectral}), addressing an enumeration question raised by the absorbing-set framework of Morris, Pllaha, and Kelley~\cite{morris2024}; and, applying the bipartite expander mixing lemma through the second-largest singular value of $H_X$ obtained from the same Fourier matrices, we show a spectral lower bound on the stopping distance of a QC-GHP code (Proposition~\ref{prop:stopping-spectral}). All of these expressions are computable in closed form from the polynomial data of the lifted-product construction, and together they specialise and unify the classical single-block spectral methods of~\cite{SF2009, karimi2012, fossorier2004, wangdolecekwesel2012} to the quantum two-block setting.

\section{Preliminaries}\label{sec:prelim}
This section the needed graph-theoretic, coding-theoretic, quantum, and spectral background used throughout the paper.
\subsection{Graph theory background}\label{sec:graph-background}
All graphs are finite and simple. For a graph $G = (V, E)$  with vertex set $V$ and edge set $E$, we write $N(v) = \{u : \{u, v\} \in E\}$ for the neighbourhood of a vertex, $\deg(v) = d_v = |N(v)|$ for its degree, and $N(S) = \bigcup_{v \in S} N(v)$. For $S \subseteq V$, the induced subgraph on $S$ is $G[S] = G_S$; when $D$ is the vertex set of a substructure such as an absorbing set (Definition~\ref{def:absorbing}) we abbreviate $\mathcal{G}[D]$ by $\mathcal{G}_D$. The \emph{girth}$g(G)$ is the length of a shortest cycle (with $g(G) = \infty$ if $G$ is acyclic). A graph is \textit{$d$-regular} if every vertex has degree $d$; a bipartite graph $G = (V_L, V_R, E)$, in which every edge joins $V_L$ to $V_R$, is \textit{$(d_L, d_R)$-biregular} if all left vertices have degree $d_L$ and all right vertices degree $d_R$. A bipartite graph has no odd cycle, so its girth, if finite, is an even integer at least $4$. We refer to \cite{brouwer2011spectra} for standard graph-theoretic background.

\begin{defn}[Adjacency and biadjacency matrices]\label{def:adj}
Let $G$ be a graph with vertex set $V = \{v_1, \ldots, v_n\}$. The \textit{adjacency matrix} of $G$ is the $n \times n$ symmetric $0$-$1$ matrix $A_G$ defined by $(A_G)_{ij} = 1$ iff $\{v_i, v_j\} \in E$. For a bipartite graph $G = (V_L, V_R, E)$ with $|V_L| = m$, $|V_R| = n$, the \textit{biadjacency matrix} $H \in \{0, 1\}^{m \times n}$ has $H_{ij} = 1$ iff the $i$-th left vertex is adjacent to the $j$-th right vertex. The full adjacency matrix is then the block matrix
\[
    A_G \;=\; \begin{pmatrix} 0 & H \\ H^T & 0 \end{pmatrix}.
\]
\end{defn}

\begin{thm}[Walk counts via matrix powers, {\cite{brouwer2011spectra}}]\label{thm:tr-walks}
For any graph $G$ with adjacency matrix $A_G$ and any integer $k \geq 0$, the entry $(A_G^k)_{uv}$ counts walks of length $k$ from vertex $u$ to vertex $v$ in $G$. Consequently,
\[
    \operatorname{tr}(A_G^k) \;=\; \#\{\text{closed walks of length $k$ in $G$}\}.
\]
\end{thm}

The bipartite specialisation of Theorem~\ref{thm:tr-walks} through the biadjacency matrix is stated as Lemma~\ref{lem:bipartite-trace}.

\subsection{Coding theory background}\label{sec:coding-background}
Let $\mathbb{F}_2 = \{0, 1\}$. For $x \in \mathbb{F}_2^n$, the \textit{support} is $\mathrm{supp}(x) = \{i : x_i = 1\}$ and the \textit{Hamming weight} is $\mathrm{wt}(x) = |\mathrm{supp}(x)|$. A \textit{(binary) linear code} $\mathcal{C} \subseteq \mathbb{F}_2^n$ of dimension $k$ (parameters $[n, k]$, rate $k/n$) is the kernel $\{x \in \mathbb{F}_2^n : Hx = 0\}$ of a \textit{parity-check matrix} $H \in \mathbb{F}_2^{m \times n}$ (with $m \geq n - k$); its \textit{minimum distance} $d = \min\{\mathrm{wt}(x) : x \in \mathcal{C},\, x \neq 0\}$ equals the smallest size of a nonzero codeword support, and we write the parameters as $[n, k, d]$. The definition extends verbatim over any finite field $\mathbb{F}_q$; all codes here are binary (in the quantum case via the symplectic representation of Pauli operators, see \S\ref{subsec:css}).

\begin{defn}[LDPC code]\label{def:ldpc}
A \textit{low-density parity-check (LDPC) code} is a linear code that admits a sparse parity-check matrix~$H$, meaning each row and each column of $H$ has a small (constant or slowly-growing) number of nonzero entries relative to the block length $n$.
\end{defn}

A parity-check matrix has \textit{column weight}~$\gamma$ (resp.\ \textit{row weight}~$k$) if every column (resp.\ row) contains exactly $\gamma$ (resp.\ $k$) ones; the code is then \textit{$(\gamma, k)$-regular}. Sparsity is asymptotic and has no sharp threshold at finite block length.

\begin{defn}[Tanner graph, \cite{1056404}]\label{def:tanner}
Given a parity check matrix $H$ of a linear code $\mathcal{C}$, its bipartite \textit{Tanner graph representation} is the graph $\mathcal{G} = (V,W; E)$
where the vertex sets $V$ and $W$ correspond to the codeword coordinates and the parity check equations, and are called \textit{variable} and \textit{check} nodes, respectively. $E$ is the set of edges. For $v_i \in V$ and $c_j \in W$, the edge $(v_i,c_j)\in E$ if and only if $h_{j,i} = 1$ in $H$.
\end{defn}

A code does not have a unique Tanner graph, as $\mathcal{G}$ depends on the choice of parity-check matrix. Being bipartite, $\mathcal{G}$ has even girth ($\geq 4$) whenever finite. A code's parameters are reflected in $\mathcal{G}$: $n = |V|$, $m = |W| \geq n - k$ (with equality when the rows of $H$ are independent), and $d$ equals the smallest size of a nonempty set $S \subseteq V$ such that every check in $N(S)$ has even degree in $\mathcal{G}[S \cup N(S)]$.

\subsubsection{Critical substructures: stopping, trapping, and absorbing sets}\label{subsec:critical-substructures}

Throughout this subsection, $\mathcal{G} = (V, W; E)$ is the Tanner graph of a binary linear code $\mathcal{C}$ with parity-check matrix $H$, and degrees are taken in $\mathcal{G}$. All definitions are purely graph-theoretic and extend verbatim to non-binary linear codes.

\begin{defn}[Stopping set, \cite{di2002}]\label{def:stopping}
A \textit{stopping set} $S$ in a Tanner graph $\mathcal{G}=(V,W;E)$ is a subset $S\subseteq V$ of variable nodes such that every check node in the neighborhood $N(S)$ is connected to $S$ by at least two edges. Equivalently, the subgraph induced by $S\cup N(S)$ contains no check node of degree one.
\end{defn}

On the binary erasure channel, iterative decoding fails to recover the codeword iff the set of erased positions contains a stopping set~\cite{di2002, schwartz2006}; the \textit{stopping distance} is the size of the smallest non-empty stopping set. See Figure~\ref{fig:substructures-summary}(a).

\begin{defn}[$(a,b)$-trapping set, \cite{richardson2003}]\label{def:trapping}
An \textit{$(a,b)$-trapping set} $T$ in a Tanner graph $\mathcal{G}$ is a subset of $a$ variable nodes whose induced subgraph $\mathcal{G}_T$ contains exactly $b$ odd-degree check nodes.
\end{defn}

Figure~\ref{fig:substructures-summary}(b) shows a $(2,2)$-trapping set $T=\{v_1,v_2\}$ with its two odd-degree checks drawn in red.

\begin{defn}[Elementary, leafless, and non-elementary trapping sets, \cite{dehghan2019}]\label{def:elementary-ts}
A trapping set is called \textit{elementary} (ETS) if every check node in its induced subgraph has degree one or two. An elementary trapping set is \textit{leafless} (LETS) if, additionally, every variable node is adjacent to at least two degree-two (satisfied) check nodes. An ETS that is not leafless is called an \textit{ETS with leaf} (ETSL). A trapping set that is not elementary is called a \textit{non-elementary trapping set} (NETS).
\end{defn}

A related global construction is the \emph{variable-node graph}, attached to the Tanner graph itself rather than to one of its substructures.

\begin{defn}[Variable-node graph]\label{def:vng}
The \emph{variable-node graph} of a Tanner graph $\mathcal{G} = (V, W; E)$, denoted $G^{VN}$, is the simple graph on the variable set $V$ in which $u, v \in V$ are adjacent iff there exists a check $c \in W$ with $\{u, c\}, \{v, c\} \in E$. When $g(\mathcal{G}) \geq 6$, any such pair $\{u, v\}$ shares at most one common check, so $G^{VN}$ has no multi-edges.
\end{defn}

Absorbing sets refine trapping sets by adding a local stability condition: every variable node in the set has strictly more satisfied (even-degree) than unsatisfied (odd-degree) check neighbours in the induced subgraph~\cite{5174520, 5361488}. This local majority is what makes absorbing sets ``absorb'' the decoder~\cite{dolecek2010}.

\begin{defn}[$(a,b)$-absorbing set, \cite{5174520, 5361488}]\label{def:absorbing}
Let $\mathcal{G}$ be a Tanner graph and $\mathcal{A} \subseteq V$ a subset of $a$ variable nodes. For each variable $v \in \mathcal{A}$, write $e_v$ (resp.\ $o_v$) for the number of check neighbours of $v$ in $\mathcal{G}_{\mathcal{A}}$ whose degree in $\mathcal{G}_{\mathcal{A}}$ is even (resp.\ odd). The set $\mathcal{A}$ is an \emph{$(a, b)$-absorbing set} if exactly $b$ check nodes in $\mathcal{G}_{\mathcal{A}}$ have odd degree and $e_v > o_v$ for every $v \in \mathcal{A}$. An absorbing set is \emph{elementary} if every check node in $\mathcal{G}_{\mathcal{A}}$ has degree $1$ or $2$ (equivalently, the underlying $(a,b)$-trapping set is an ETS in the sense of Definition~\ref{def:elementary-ts}).
\end{defn}

\begin{figure}[ht]
\centering
\tikzset{
  varN/.style={circle, draw=black, fill=white, minimum size=8pt, inner sep=0pt},
  varIn/.style={circle, draw=black, fill=ugentblue, minimum size=8pt, inner sep=0pt},
  chkN/.style={regular polygon, regular polygon sides=4, draw=black, fill=white, minimum size=8pt, inner sep=0pt},
  chkOdd/.style={regular polygon, regular polygon sides=4, draw=red, fill=red!25, minimum size=10pt, inner sep=0pt, very thick},
  hilEdge/.style={thick, ugentblue},
  oddEdge/.style={thick, red},
  skelEdge/.style={black!30}
}
\begin{subfigure}[t]{0.32\textwidth}
\centering
\begin{tikzpicture}[thick, scale=0.7, every node/.style={transform shape}]
\node[varIn, label=left:{\scriptsize $v_1$}] (v1) at (0, 2.4) {};
\node[varIn, label=left:{\scriptsize $v_2$}] (v2) at (0, 0.8) {};
\node[varIn, label=left:{\scriptsize $v_3$}] (v3) at (0,-0.8) {};
\node[varN,  label=left:{\scriptsize $v_4$}] (v4) at (0,-2.4) {};
\node[chkN, label=right:{\scriptsize $c_1$}] (c1) at (3, 2.4) {};
\node[chkN, label=right:{\scriptsize $c_2$}] (c2) at (3, 0.8) {};
\node[chkN, label=right:{\scriptsize $c_3$}] (c3) at (3,-0.8) {};
\node[chkN, label=right:{\scriptsize $c_4$}] (c4) at (3,-2.4) {};
\draw[hilEdge] (v1)--(c1); \draw[hilEdge] (v1)--(c2); \draw[hilEdge] (v1)--(c3);
\draw[hilEdge] (v2)--(c1); \draw[hilEdge] (v2)--(c2); \draw[hilEdge] (v2)--(c4);
\draw[hilEdge] (v3)--(c1); \draw[hilEdge] (v3)--(c3); \draw[hilEdge] (v3)--(c4);
\draw[skelEdge] (v4)--(c2); \draw[skelEdge] (v4)--(c3); \draw[skelEdge] (v4)--(c4);
\end{tikzpicture}
\caption{Stopping set $S=\{v_1,v_2,v_3\}$: every check in $N(S)$ has $\geq 2$ neighbours in~$S$.}
\end{subfigure}
\hfill
\begin{subfigure}[t]{0.32\textwidth}
\centering
\begin{tikzpicture}[thick, scale=0.7, every node/.style={transform shape}]
\node[varIn, label=left:{\scriptsize $v_1$}] (v1) at (0, 2.4) {};
\node[varIn, label=left:{\scriptsize $v_2$}] (v2) at (0, 0.8) {};
\node[varN,  label=left:{\scriptsize $v_3$}] (v3) at (0,-0.8) {};
\node[varN,  label=left:{\scriptsize $v_4$}] (v4) at (0,-2.4) {};
\node[chkN,   label=right:{\scriptsize $c_1$}] (c1) at (3, 2.4) {};
\node[chkN,   label=right:{\scriptsize $c_2$}] (c2) at (3, 0.8) {};
\node[chkOdd, label=right:{\scriptsize $c_3$}] (c3) at (3,-0.8) {};
\node[chkOdd, label=right:{\scriptsize $c_4$}] (c4) at (3,-2.4) {};
\draw[hilEdge] (v1)--(c1); \draw[hilEdge] (v1)--(c2); \draw[oddEdge] (v1)--(c3);
\draw[hilEdge] (v2)--(c1); \draw[hilEdge] (v2)--(c2); \draw[oddEdge] (v2)--(c4);
\draw[skelEdge] (v3)--(c1); \draw[skelEdge] (v3)--(c3); \draw[skelEdge] (v3)--(c4);
\draw[skelEdge] (v4)--(c2); \draw[skelEdge] (v4)--(c3); \draw[skelEdge] (v4)--(c4);
\end{tikzpicture}
\caption{$(2,2)$-trapping and absorbing set $T=\{v_1,v_2\}$: $c_3,c_4$ are odd-degree in $\mathcal{G}_T$.}
\end{subfigure}
\hfill
\begin{subfigure}[t]{0.32\textwidth}
\centering
\begin{tikzpicture}[thick, scale=0.7, every node/.style={transform shape}]
\node[varIn, label=left:{\scriptsize $v_1$}] (v1) at (0, 2.4) {};
\node[varIn, label=left:{\scriptsize $v_2$}] (v2) at (0, 0.8) {};
\node[varIn, label=left:{\scriptsize $v_3$}] (v3) at (0,-0.8) {};
\node[varIn, label=left:{\scriptsize $v_4$}] (v4) at (0,-2.4) {};
\node[chkOdd, label=right:{\scriptsize $c_1$}] (c1) at (3, 2.4) {};
\node[chkOdd, label=right:{\scriptsize $c_2$}] (c2) at (3, 0.8) {};
\node[chkOdd, label=right:{\scriptsize $c_3$}] (c3) at (3,-0.8) {};
\node[chkOdd, label=right:{\scriptsize $c_4$}] (c4) at (3,-2.4) {};
\draw[oddEdge] (v1)--(c1); \draw[oddEdge] (v1)--(c2); \draw[oddEdge] (v1)--(c3);
\draw[oddEdge] (v2)--(c1); \draw[oddEdge] (v2)--(c2); \draw[oddEdge] (v2)--(c4);
\draw[oddEdge] (v3)--(c1); \draw[oddEdge] (v3)--(c3); \draw[oddEdge] (v3)--(c4);
\draw[oddEdge] (v4)--(c2); \draw[oddEdge] (v4)--(c3); \draw[oddEdge] (v4)--(c4);
\end{tikzpicture}
\caption{$(4,4)$-trapping set $\{v_1,\dots,v_4\}$ that is \emph{not} absorbing: each $v_i$ has $e_{v_i}=0<3=o_{v_i}$.}
\end{subfigure}
\caption{Three substructures on the same $(3,3)$-regular Tanner graph. Variables in the highlighted set are filled blue; check nodes with odd induced degree are drawn in red. Edges of the induced subgraph $\mathcal{G}_S$ (resp.\ $\mathcal{G}_T$, $\mathcal{G}_{\mathcal{A}}$) are coloured; edges incident to excluded variables are faded. The substructure classes form a strict hierarchy: $(a)$ is also a $(3,1)$-trapping and $(3,1)$-absorbing set; $(b)$ is trapping and absorbing but not stopping (because $c_3,c_4$ have induced degree~$1$); $(c)$ is trapping but neither stopping nor absorbing.}
\label{fig:substructures-summary}
\end{figure}
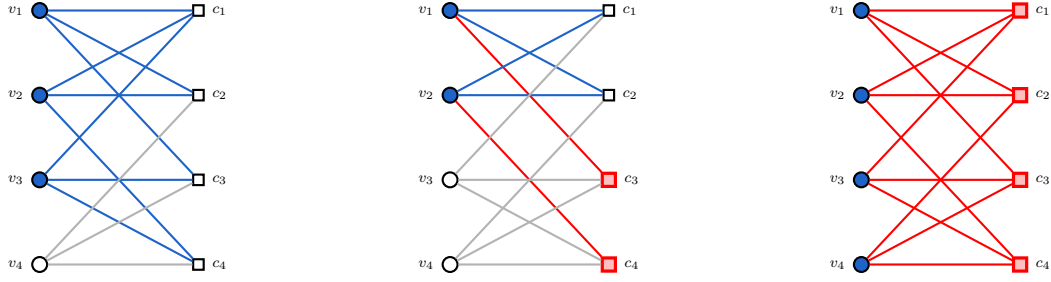

\subsection{Quantum LDPC codes}\label{sec:quantum-background}
Quantum error-correcting codes encode $k$ logical qubits into $n$ physical qubits, embedding the logical state into a structured subspace of $(\mathbb{C}^2)^{\otimes n}$; see Breuckmann and Eberhardt~\cite{breuckmann2021quantumldpcsurvey} for a survey. We work with the \textit{stabilizer codes} of Gottesman~\cite{gottesman1997} and Calderbank-Rains-Shor-Sloane~\cite{CRSS1997}, and specifically with the Calderbank-Shor-Steane (CSS) subclass, whose stabilizer factors into pure-$X$ and pure-$Z$ generators. In the $n$-qubit Pauli group any two elements commute or anticommute, and commutation is governed by the \textit{symplectic inner product} on $\mathbb{F}_2^n$~\cite{CRSS1997, gottesman1997}; this is the structural reason CSS codes admit a binary-matrix description. We focus on the \textit{QLDPC} subclass (Definition~\ref{def:qldpc}), whose running example is the \textit{quasi-cyclic generalized hypergraph product} (QC-GHP) family of~\cite{panteleev2021degenerate}.

\subsubsection{CSS codes}\label{subsec:css}
A \textit{stabilizer code} is the joint $+1$ eigenspace of an abelian subgroup $\mathcal{S} \subseteq \mathcal{P}_n$ of the Pauli group~\cite{gottesman1997, CRSS1997}; we work directly with the matrix description of the stabilizer.

\begin{defn}[CSS code, \cite{CRSS1997, gottesman1997}]\label{def:css}
A \textit{Calderbank-Shor-Steane (CSS) code} is a stabilizer code whose stabilizer admits a set of generators that are either pure-$X$ (tensor products of $I$ and $X$ only) or pure-$Z$. When the pure-$X$ generators are specified by a matrix $H_X \in \mathbb{F}_2^{m_X \times n}$ (each row encodes the support of one generator) and the pure-$Z$ generators by a matrix $H_Z \in \mathbb{F}_2^{m_Z \times n}$, the commutation requirement on the stabilizer translates to the orthogonality condition
\begin{equation}\label{eq:css-commutation}
    H_X H_Z^{\mkern1mu T} \;=\; 0 \qquad \text{(over } \mathbb{F}_2\text{).}
\end{equation}
The number of logical qubits encoded by the CSS code is
\[
    k \;=\; n - \mathrm{rank}_{\mathbb{F}_2}(H_X) - \mathrm{rank}_{\mathbb{F}_2}(H_Z),
\]
and the \textit{minimum distance} $d$ is the smallest weight of a Pauli operator that commutes with every stabilizer generator but does not lie in $\mathcal{S}$. The code is then said to have parameters $[\![n, k, d]\!]$, where the double brackets distinguish quantum from classical parameters.
\end{defn}

Equivalently, the row spaces of $H_X$ and $H_Z$ are mutually orthogonal, so the classical codes $\mathcal{C}_X = \ker H_X$ and $\mathcal{C}_Z = \ker H_Z$ form a nested pair $\mathcal{C}_X^\perp \subseteq \mathcal{C}_Z$, and the CSS construction recovers a quantum code from any such pair.

\subsubsection{QLDPC codes}\label{subsec:qldpc}
\begin{defn}[QLDPC code]\label{def:qldpc}
A \textit{quantum low-density parity-check (QLDPC) code} is a CSS code in which both parity-check matrices $H_X$ and $H_Z$ are sparse in the sense of Definition~\ref{def:ldpc}.
\end{defn}

QLDPC codes~\cite{mackay2004sparse, kovalev2013} have become a central object in fault-tolerant quantum computing, with asymptotically good constructions due to Panteleev-Kalachev~\cite{panteleev2022asymptotically} and Leverrier-Z{\'e}mor~\cite{leverrier2022quantum}. As classically, sparsity is asymptotic; in structured families the column and row weights of $H_X, H_Z$ are fixed constants. Each CSS code carries two Tanner graphs, $\mathcal{G}_X$ from $H_X$ and $\mathcal{G}_Z$ from $H_Z$ (Definition~\ref{def:tanner}); these are the principal objects on which the decoding analyses in this work take place.

\subsubsection{Syndrome-based decoding}\label{subsec:syndrome}
When a Pauli error $E = X^{e_X} Z^{e_Z}$ corrupts the encoded state, measuring the $X$-type generators yields the \textit{$X$-syndrome} $H_X e_Z \in \mathbb{F}_2^{m_X}$ and measuring the $Z$-type generators yields the \textit{$Z$-syndrome} $H_Z e_X$. \textit{Syndrome-based decoding} produces low-weight estimates $\hat e_X, \hat e_Z$ consistent with the observed syndromes~\cite{raveendran2019}, run block-by-block with one iterative message-passing decoder per Tanner graph~\cite{raveendran2019, raveendran2021}. Decoding succeeds when $\hat e_X - e_X \in \mathrm{rowspace}(H_X)$ (a \textit{degenerate error} acting trivially on the codespace) and produces a \textit{logical error} otherwise.

The graph-theoretic substructures responsible for these failures are the quantum analogues of the stopping, trapping and absorbing sets of \S\ref{subsec:critical-substructures}, taken up in Section~\ref{sec:background}. The verbatim graph-theoretic definitions (Definitions~\ref{def:stopping}, \ref{def:trapping}, \ref{def:absorbing}) carry over unchanged, since $\mathcal{G}_X$ and $\mathcal{G}_Z$ are themselves classical bipartite graphs; what differs is that the decoder input is a measured syndrome rather than a corrupted codeword, so the substructures are characterised through the syndrome-matching analysis of Morris-Pllaha-Kelley~\cite{morris2024}. In particular the \textit{$b = 0$} case acquires a genuinely quantum interpretation as a symmetric-stabilizer configuration on which iterative decoding oscillates rather than diverges~\cite{raveendran2021, morris2024}.

\subsection{Spectral and Fourier background}\label{sec:spectral-background}
We provide the discrete Fourier and circulant-matrix tools that will be used in the proposed spectral cycle-counting framework. Throughout, $\ell \geq 1$ is a positive integer and $\omega = e^{2\pi i / \ell}$ is a primitive $\ell$-th root of unity.

\subsubsection{Discrete Fourier transform on $\mathbb{Z}_\ell$}\label{subsec:dft}
The \textit{characters} of $\mathbb{Z}_\ell$ are $\chi_j(k) = \omega^{jk}$, and the \textit{discrete Fourier transform} (DFT) of $f: \mathbb{Z}_\ell \to \mathbb{C}$ is $\hat f(j) = \sum_{k=0}^{\ell-1} f(k)\,\omega^{-jk}$, with inverse $f(k) = \tfrac{1}{\ell}\sum_j \hat f(j)\,\omega^{jk}$~\cite[Ch.~1]{davis1979circulants}. The \textit{Parseval (energy) identity}~\cite[\S 3.2]{davis1979circulants} reads
\begin{equation}\label{eq:parseval}
    \sum_{k = 0}^{\ell - 1} |f(k)|^2 \;=\; \frac{1}{\ell}\sum_{j = 0}^{\ell - 1} |\hat f(j)|^2.
\end{equation}
We apply~\eqref{eq:parseval} to the symbol polynomials of circulant generator matrices, converting a time-domain sum into a sum over frequency-domain eigenvalues.

\subsubsection{Circulant matrices and block-diagonalization}\label{subsec:circulants}
A matrix $C \in \mathbb{C}^{\ell \times \ell}$ is \textit{circulant} if $C_{ij} = c_{(i - j) \bmod \ell}$ for some \textit{symbol} $c \in \mathbb{C}^\ell$; equivalently $C = \sum_{k=0}^{\ell-1} c_k\,S^k$ where $S$ is the cyclic shift matrix~\cite[\S 2.1]{davis1979circulants}. The Fourier matrix $F$ with entries $F_{jk} = \omega^{jk} / \sqrt{\ell}$ is unitary and simultaneously diagonalises every circulant~\cite[Thm.~3.2.1]{davis1979circulants}:
\begin{equation}\label{eq:circulant-diag}
    F\, C\, F^* \;=\; \mathrm{diag}\bigl(\hat c(0), \hat c(1), \ldots, \hat c(\ell - 1)\bigr),
\end{equation}
where $\hat c(j) = \sum_k c_k \omega^{jk}$; the same conjugation blockwise-diagonalises any block matrix of $\ell \times \ell$ circulants (replace $F$ by $F \otimes I_r$ for blocks of size $r$). Quasi-cyclic Tanner graphs, including the QC-GHP family of Example~\ref{ex:ghp}, have circulant-block biadjacency matrices, and~\eqref{eq:circulant-diag} reduces their global cycle-counting problems to a sum over $\ell$ smaller matrices $\{M_j\}_{j=0}^{\ell - 1}$, the \textit{Fourier matrices} of the protograph: this is the structural basis of Theorem~\ref{prop:spectral-trace}.

\subsubsection{Cycle counts from spectral traces}\label{subsec:trace-cycle-counts}
Theorem~\ref{thm:tr-walks} specialises to bipartite graphs through the biadjacency matrix: the powers of $HH^T$ count closed walks that begin and end on the same side.

\begin{lem}[Side-restricted trace identity; bipartite specialisation of Theorem~\ref{thm:tr-walks}, {\cite{brouwer2011spectra}}]\label{lem:bipartite-trace}
Let $G = (V_L, V_R, E)$ be a bipartite graph with biadjacency matrix $H \in \{0,1\}^{|V_L| \times |V_R|}$. For every integer $k \geq 1$ and every $u \in V_L$, the diagonal entry $\bigl((HH^T)^k\bigr)_{uu}$ equals the number of closed walks of length $2k$ in $G$ that begin and end at $u$. Summing over $u \in V_L$,
\[
    \operatorname{tr}\!\bigl((HH^T)^k\bigr) \;=\; \#\{\text{closed walks of length $2k$ in $G$ starting in $V_L$}\}.
\]
The analogous identity holds for $(H^T H)^k$ and walks starting in $V_R$; by the cyclic invariance of trace the two counts are equal.
\end{lem}

\begin{proof}
From Definition~\ref{def:adj} we have $$A_G^{2} = \begin{pmatrix} HH^T & 0 \\ 0 & H^T H \end{pmatrix},$$ so $A_G^{2k}$ is block-diagonal with blocks $(HH^T)^k$ and $(H^T H)^k$. Therefore $\bigl((HH^T)^k\bigr)_{uu} = \bigl(A_G^{2k}\bigr)_{uu}$ for $u \in V_L$, and Theorem~\ref{thm:tr-walks} identifies the right-hand side as the number of closed walks of length $2k$ at $u$. The remaining claims follow by summation and from $\operatorname{tr}\bigl((HH^T)^k\bigr) = \operatorname{tr}\bigl((H^T H)^k\bigr)$.
\end{proof}

Closed-walk counts overcount cycles because they include back-tracking steps and tree-shaped excursions; subtracting these via M\"obius-type corrections converts $\operatorname{tr}\bigl((HH^T)^k\bigr)$ into a closed-form count of \textit{cycles} of length $2k$. The precise $4$-cycle formula for quasi-cyclic Tanner graphs, expressed as a sum $\sum_{j=0}^{\ell-1} \operatorname{tr}\bigl((M_j M_j^*)^2\bigr)$ over the Fourier matrices of~\eqref{eq:circulant-diag}, is derived in Theorem~\ref{prop:spectral-trace}; classical cycle-counting precursors appear in~\cite{karimi2012, SF2009}.

\section{Lifted-product QLDPC codes and the absorbing-set framework}\label{sec:background}

Quantum LDPC (QLDPC) codes are among the leading current candidates for scalable fault-tolerant quantum computation~\cite{gottesman2014, panteleev2022asymptotically, leverrier2022quantum}. As in the classical case, their finite-length performance is determined not by asymptotics but by a small number of stable graphical configurations in the Tanner graph that cause persistent decoder failure~\cite{richardson2003, 5361488}, and these critical substructures are far less understood in the quantum setting~\cite{chytas2025collective, stambler2023addressing, morris2024, 7350103} than in the classical one~\cite{dehghan2019, mcmillon2023, AHK2019, 1023274, orlitsky2005stopping}. A systematic \emph{spectral} graph-theoretic approach to QLDPC Tanner graphs, parallel to the adjacency-spectrum cycle counts used extensively for classical LDPC codes~\cite{dehghan2018, karimi2012, blake2011LDPC}, has not yet appeared: the CSS commutation constraint and the lifted-product block structure pose additional obstructions. This work takes a first step via a discrete Fourier transform (DFT)-based cycle-counting framework for lifted-product Tanner graphs.

\subsection{QLDPC codes and the absorbing-set framework}\label{subsec:quantum}

 Both structural threads to date concern \emph{syndrome-based iterative decoding} (\S\ref{subsec:syndrome}), the natural decoding model for stabilizer codes~\cite{raveendran2019}. Raveendran and Vasi\'c~\cite{raveendran2021} identify \emph{symmetric stabilizers}, a genuinely quantum failure mode in which decoder degeneracy lets belief propagation oscillate between equivalent error estimates rather than converge. Morris, Pllaha, and Kelley~\cite{morris2024} show that \emph{absorbing sets} (Definition~\ref{def:absorbing}) are the dominant failure mechanism for syndrome-based iterative decoders: any $(a,b)$-absorbing set with $b\ge 1$ is failure-inducing (its syndrome estimate stays $\bm{0}$, permanently mismatching the input syndrome at the $b$ odd-degree checks), while the $b=0$ case subsumes the symmetric-stabilizer phenomenon and governs the degenerate-versus-logical-error dichotomy~\cite{morris2024}. For CSS hypergraph-product codes~\cite{tillich2014} the absorbing sets are completely characterised by the base matrices $H_1$, $H_2$ and their transposes, so small base-code absorbing sets propagate to low-weight failure-inducing sets~\cite{morris2024}.

Absorbing sets thus give a single combinatorial object subsuming the symmetric-stabilizer formalism of~\cite{raveendran2021}, which is why subsequent QLDPC structural analyses~\cite{stambler2023addressing, chytas2025collective, morris2024} have converged on them. Moreover, the classical Hoory-type cage and girth lower bounds on absorbing-set sizes~\cite{hoory2002, mcmillon2023} do not transfer to the QLDPC two-block form $H_X = [\widetilde{H}_1 \mid I_{r_1} \otimes \widetilde{B}^T]$, because the CSS commutation constraint $H_X H_Z^T = 0$ forces girth at most~$6$ in any fully connected quasi-cyclic construction with column weight at least~$3$~\cite{AmirzadePanarioSadeghi2024}, foreclosing the high-girth regime in which Hoory-type counting gives strong bounds.

\subsection{Lifted-product QLDPC codes}\label{sec:liftedproduct}

We review the lifted-product construction of Panteleev and Kalachev~\cite{panteleev2021degenerate} that produces the running QC-GHP example and the two-block form $H_X = [\widetilde{H}_1 \mid I_{r_1} \otimes \widetilde{B}^T]$ above.

\begin{defn}\cite{panteleev2021degenerate}\label{def:liftedproduct} Consider $H_1 \in \mathbb{F}_2^{r_1 \times n_1}$, $H_2 \in \mathbb{F}_2^{r_2 \times n_2}$, and $G$ a finite abelian group of size $L$. Each $g \in G$ is associated with an $L \times L$ permutation matrix $P_g$ over $\mathbb{F}_2$ using the regular representation, defined by $(P_g)_{h,k}=1$ if $h=gk$; then $P_g P_{g'} = P_{gg'}$ and $P_{g^{-1}} = P_g^T$.

Given a labeling of the nonzero entries of $H_1$ and $H_2$ by elements of $G$, we regard $H_1$ and $H_2$ as matrices over the group algebra $R = \mathbb{F}_2[G]$. For a matrix $M$ over $R$, write $\widetilde{M}$ for its binary \textit{lift}, obtained by replacing each entry $\sum_g a_g\, g$ with the $L \times L$ matrix $\sum_g a_g P_g$ (so $\tilde{H}_1 \in \mathbb{F}_2^{(r_1L) \times (n_1L)}$), and $M^*$ for its conjugate transpose, $(M^*)_{ij} = (M_{ji})^*$ with $\bigl(\sum_g a_g\, g\bigr)^* = \sum_g a_g\, g^{-1}$; on lifts, $\widetilde{M^*} = \widetilde{M}^{\,T}$. The \textit{lifted product quantum CSS code} is defined by its parity check matrices
$$H_X = \bigl[\, \widetilde{H_1 \otimes I_{n_2}} \;\;\big|\;\; \widetilde{I_{r_1} \otimes H_2^*} \,\bigr] \qquad \text{and} \qquad H_Z = \bigl[\, \widetilde{I_{n_1} \otimes H_2} \;\;\big|\;\; \widetilde{H_1^* \otimes I_{r_2}} \,\bigr],$$
where the Kronecker products are taken over $R$; the commutation condition~\eqref{eq:css-commutation} holds because $H_X H_Z^T$ is the lift of $H_1 \otimes H_2^* + H_1 \otimes H_2^* = 0$ over $R$. (For nonabelian $G$ the second factor must instead act through the right regular representation, see~\cite{panteleev2022asymptotically}; we do not need that generality here.) In the case $r_2 = n_2 = 1$ with $H_2 = (b)$, used throughout this paper, this reduces to $H_X = [\tilde{H}_1 \mid I_{r_1} \otimes \tilde{H}_2^T]$ and $H_Z = [I_{n_1} \otimes \tilde{H}_2 \mid \tilde{H}_1^T]$.
\end{defn}

\begin{ex}[Generalized bicycle codes as lifted products]\label{ex:gb}
Take $r_1 = n_1 = r_2 = n_2 = 1$, so $H_1 = (1)$ and $H_2 = (1)$ are $1 \times 1$ matrices over $\mathbb{F}_2$. Let $G = \mathbb{Z}_\ell = \langle x \mid x^\ell = 1 \rangle$ be the cyclic group of order~$\ell$. Label the single nonzero entry of $H_1$ by the group-algebra element $a(x) \in \mathbb{F}_2[x]/(x^\ell - 1)$, and the single nonzero entry of $H_2$ by the reversed polynomial $b(x^{-1})$, so that $H_2^* = (b(x))$. Writing $A$ and $B$ for the $\ell \times \ell$ circulant matrices of $a(x)$ and $b(x)$, the parity check matrices of the lifted product code become
\[
  H_X = [A \mid B], \qquad H_Z = [B^\top \mid A^\top].
\]
This is precisely the generalized bicycle (GB) ansatz of Kovalev and Pryadko~\cite{kovalev2013}. For instance, taking $\ell = 24$, $a(x) = 1 + x^2 + x^8 + x^{15}$, and $b(x) = 1 + x^2 + x^{12} + x^{17}$ yields the $[\![48, 6, 8]\!]$ code (code~A3 in~\cite{panteleev2021degenerate}).
\end{ex}

\begin{ex}[A quasi-cyclic GHP code as a lifted product]\label{ex:ghp}
Now take $r_1 = n_1 = 7$, $r_2 = n_2 = 1$, $G = \mathbb{Z}_{63}$, $H_2 = (1)$ with label $b(x) = 1 + x + x^6$, and $H_1$ a $7 \times 7$ matrix over $\mathbb{F}_2$ whose nonzero entries are labeled by monomials $x^i \in \mathbb{Z}_{63}$:
\[
  H_1 = \begin{pmatrix}
    x^{27} & 0      & 0      & 0      & 0      & x^{54} & 1      \\
    1      & x^{27} & 0      & 0      & 0      & 0      & x^{54} \\
    x^{54} & 1      & x^{27} & 0      & 0      & 0      & 0      \\
    0      & x^{54} & 1      & x^{27} & 0      & 0      & 0      \\
    0      & 0      & x^{54} & 1      & x^{27} & 0      & 0      \\
    0      & 0      & 0      & x^{54} & 1      & x^{27} & 0      \\
    0      & 0      & 0      & 0      & x^{54} & 1      & x^{27}
  \end{pmatrix}.
\]
Lifting replaces each monomial~$x^i$ with the corresponding $63 \times 63$ cyclic permutation matrix and each~$0$ with the zero matrix; the single entry $b(x) = 1 + x + x^6$ of $H_2$ lifts to the $63 \times 63$ circulant $\tilde{H}_2 = B$. The result is a CSS code with
\[
  H_X = [\tilde{H}_1 \mid I_7 \otimes B^\top], \qquad H_Z = [I_7 \otimes B \mid \tilde{H}_1^\top],
\]
which is the $[\![882, 24, d]\!]$ generalized hypergraph product code (code~B1 in~\cite{panteleev2021degenerate}) with $18 \leq d \leq 24$. (In~\cite{panteleev2021degenerate} code~B1 is presented with the labels $1$ and $x^{54}$ interchanged and with $H_X = [\tilde{H}_1 \mid I_7 \otimes B]$, $H_Z = [I_7 \otimes B^\top \mid \tilde{H}_1^\top]$; the substitution $x \mapsto x^{-1}$ combined with monomial rescalings identifies that presentation with the one used here, so both define the same code.) Its Tanner graph is $(3,6)$-regular and has girth~$6$; the protograph of $H_1$ is depicted in Figure~\ref{fig:ghp-protograph}.
\end{ex}

\begin{figure}[ht]
\centering
\begin{tikzpicture}[
  lnode/.style={circle, draw, fill=black, inner sep=2pt, minimum size=5pt},
  rnode/.style={circle, draw, fill=white, inner sep=2pt, minimum size=5pt},
  edgeA/.style={thick, ugentblue},
  edgeB/.style={thick, red, dashed},
  edgeC/.style={thick, black!50!green, dotted, thick}
]

\foreach \i in {1,...,7} {
  \node[lnode, label=left:{$\ell_{\i}$}] (L\i) at (0, -\i*1.2) {};
}

\foreach \i in {1,...,7} {
  \node[rnode, label=right:{$r_{\i}$}] (R\i) at (5, -\i*1.2) {};
}

\foreach \i in {1,...,7} {
  \draw[edgeA] (L\i) -- (R\i);
}

\draw[edgeB] (L1) -- (R7);
\draw[edgeB] (L2) -- (R1);
\draw[edgeB] (L3) -- (R2);
\draw[edgeB] (L4) -- (R3);
\draw[edgeB] (L5) -- (R4);
\draw[edgeB] (L6) -- (R5);
\draw[edgeB] (L7) -- (R6);

\draw[edgeC] (L1) -- (R6);
\draw[edgeC] (L2) -- (R7);
\draw[edgeC] (L3) -- (R1);
\draw[edgeC] (L4) -- (R2);
\draw[edgeC] (L5) -- (R3);
\draw[edgeC] (L6) -- (R4);
\draw[edgeC] (L7) -- (R5);

\draw[edgeA] (6.5, -1.2) -- ++(1,0) node[right, black] {$x^{27}$};
\draw[edgeB] (6.5, -2.0) -- ++(1,0) node[right, black] {$1$};
\draw[edgeC] (6.5, -2.8) -- ++(1,0) node[right, black] {$x^{54}$};

\end{tikzpicture}
\caption{The protograph of $H_1$ for the $[\![882, 24]\!]$ GHP code of Example \ref{ex:ghp}. Left vertices $\ell_i$ represent row nodes and right vertices $r_j$ represent column nodes. Each edge is labeled by its group element in $\mathbb{Z}_{63}$; in the lift, label $x^i$ is replaced by the cyclic permutation matrix $P^i$ of size $63 \times 63$. The three edge types correspond to the three circulant diagonals of $H_1$.}
\label{fig:ghp-protograph}
\end{figure}
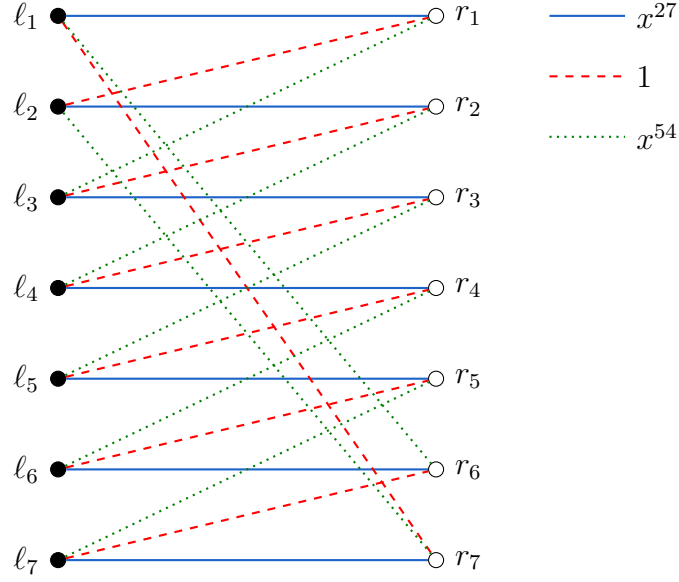

\subsection{Tanner graph properties of quasi-cyclic GHP codes}\label{sec:qcghp-properties}

\paragraph{Regularity.}
The Tanner graph of $H_X$ (and symmetrically of $H_Z$) for the $[\![882,24,d]\!]$ code is $(3,6)$-regular. More generally, when $H_1$ has constant column weight~$w_1$ and constant row weight~$w_1'$ and $H_2$ is $1 \times 1$ with label $b(x)$ of Hamming weight~$w_2$, each variable node of $H_X = [\tilde{H}_1 \mid I_{r_1} \otimes \tilde{H}_2^T]$ has degree $w_1$ or $w_2$ according to its column block, while each check node has degree $w_1' + w_2$; in Example~\ref{ex:ghp}, $w_1 = w_2 = 3$ and $w_1' + w_2 = 6$, recovering the $(3,6)$-biregular graph. The biregular pair $(\gamma,k)$ is exploited spectrally throughout: the DFT block-diagonalisation $H_X H_X^T \cong \bigoplus_j M_j$ of Section~\ref{sec:spectral-cycles} is a biregular-specific simplification, and the closed-form biregular cycle counts of Smarandache and Mitchell~\cite{SM2026} lift to the QC-GHP setting (Remark~\ref{rem:higher-Nk-pushes}).

\paragraph{Protograph and girth.}
Lifting over $G = \mathbb{Z}_\ell$ preserves vertex degree, so biregularity of the protograph implies biregularity of the Tanner graph with the same degree pair~\cite{1577834,koetter2003graph}. A protograph closed walk with $\mathbb{Z}_\ell$-labels $g_1, \dots, g_{2k}$ lifts to $\ell$ Tanner-graph cycles if and only if the alternating sum $g_1 - g_2 + \cdots - g_{2k} \equiv 0 \pmod{\ell}$, and to none otherwise~\cite{karimi2012}; this obstruction is the source of the Sidon-type arithmetic conditions on $a(x), b(x)$ in Section~\ref{sec:spectral-cycles}. The Tanner graph of the $[\![882,24,d]\!]$ code has girth~$6$~\cite{panteleev2021degenerate}, which is optimal within the QC framework: any fully connected quasi-cyclic construction with column weight at least~$3$ yielding a CSS code ($H_X H_Z^T = 0$) has girth at most~$6$~\cite{AmirzadePanarioSadeghi2024}, so the classical strategy of raising girth to suppress short cycles is structurally unavailable; girth beyond~$6$ requires column weight~$2$~\cite{AmirzadePanarioSadeghi2024}.

\paragraph{Distance, trapping, and absorbing sets.}
Panteleev and Kalachev report $18 \leq d \leq 24$ for the $[\![882,24]\!]$ code~\cite{panteleev2021degenerate}; computing quantum minimum distance is hard in general~\cite{KapshikarKundu2023}, and sharp bounds are known only for generalized bicycle codes~\cite{WangPryadko2022}. Raveendran and Vasi\'c~\cite{raveendran2021} show that symmetric-stabilizer trapping sets occur already in generalized bicycle codes (the $r_1 = n_1 = 1$ case) and that a layered decoding schedule mitigates them. For hypergraph-product codes the absorbing sets are governed by the base matrices $H_1$, $H_2$ and their transposes~\cite{morris2024}; how the group lift interacts with this structure for QC-GHP codes with circulant base matrices is identified as open in~\cite{morris2024}. The chain-complex constructions of~\cite{morris2024} exploit the symmetric tensor form $H_{X, \mathrm{HGP}} = [H_1 \otimes I \mid I \otimes H_2^T]$ and do not transfer directly to the asymmetric lifted-product form $H_X = [\widetilde{H}_1 \mid I_{r_1} \otimes \widetilde{B}^T]$; the Sidon-type characterisation of Corollary~\ref{cor:sidon} and the spectral $6$-cycle count of Remark~\ref{rem:6cycles} give algebraic handles on this interaction in the girth-$6$ regime.

\subsection{Trapping sets, absorbing sets, and stopping sets in QLDPC codes}\label{sec:qldpc-substructures}

Existing QLDPC trapping/absorbing-set work~\cite{stambler2023addressing, chytas2025collective} develops \emph{algorithmic} methods \linebreak (learning-aided trapping-set identification, decoder modifications, post-decoding correction) but yields no closed-form lower bounds on absorbing-set sizes; transferring classical biregular bounds to the two-block setting $H_X = [\widetilde{H}_1 \mid I_{r_1} \otimes \widetilde{B}^T]$ would require an analogue of the Zhao-Xiong-Ye-Yan~\cite{zhao2026turan} assumption that no two $8$-cycles share a variable node, which for QC-GHP codes becomes a joint condition on the supports of $\widetilde{H}_1$ and $b(x)$ strictly stronger than the girth-$6$ Sidon condition of Corollary~\ref{cor:sidon}. In the classical setting, stopping sets (Definition~\ref{def:stopping}) characterise the failure of iterative decoding on the binary erasure channel~\cite{di2002, schwartz2006} and the stopping distance lower-bounds the minimum distance; the QLDPC analogue carries the additional X/Z block structure, giving two coupled notions: X-stopping sets $S \subseteq V_L(G_X)$ and Z-stopping sets $S \subseteq V_L(G_Z)$, each governing iterative decoding of one block of syndromes.

\section{A spectral framework for cycle counting for lifted-product Tanner graphs}\label{sec:spectral-cycles}

In this section we adapt and extend the classical spectral cycle-counting framework of Smarandache and Flanagan~\cite{SF2009} and Karimi and Banihashemi~\cite{karimi2012} to the Tanner graphs of lifted-product QLDPC codes~\cite{panteleev2021degenerate}.
For cyclic lifts over $G = \mathbb{Z}_\ell$, the biadjacency matrix of the Tanner graph block-diagonalises under the discrete Fourier transform, reducing cycle-counting traces to a sum over $\ell$ small matrices~\cite[Thm.~1, Cor.~2]{SF2009}, \cite[Theorem~2]{karimi2012}; in our setting this enters through Theorem~\ref{prop:spectral-trace} below.
For classical QC-LDPC codes the reduction is due to Smarandache and Flanagan~\cite{SF2009}, who showed that the eigenvalues of the $r L \times r L$ matrix $H^T H$ for an $L$-fold cyclic lift of a $J \times r$ protograph reduce to those of $L$ matrices of size $r \times r$ obtained by Fourier evaluation; Karimi and Banihashemi~\cite{karimi2012} give an alternative recipe via the directed edge matrix; and Smarandache and Mitchell~\cite{SM2026} derive explicit closed forms for $N_{2k}$ in any biregular Tanner graph, valid for $k \leq 7$.

On the quantum side, Zhang~\cite{Zhang2024thesis} carries out a finite-length combinatorial cycle analysis of the lifted product construction of~\cite{panteleev2021degenerate} (and its left-right generalisation), focusing on dominant $8$-cycles and proving a girth upper bound showing that $8$-cycles are unavoidable regardless of the group-algebra elements, unlike the classical QC-LDPC case. Our framework is complementary: where Zhang enumerates $8$-cycle topologies case by case, we develop a Fourier-spectral identity (Theorem~\ref{prop:spectral-trace}) that translates short-cycle counts into traces of small matrices, applies uniformly to all even cycle lengths $2k$ up to the girth bound, and yields the additive-energy/Sidon-type characterisations of \S\ref{sec:spectral-cycles} absent from~\cite{Zhang2024thesis}.
This section specialises the classical \emph{single-block} framework to the quantum two-block structure $H_X = [\widetilde{H}_1 \mid I_{r_1} \otimes \widetilde{B}^T]$ (Theorem~\ref{prop:spectral-trace}, where the second block contributes only a scalar shift $|b(\omega^j)|^2 I_{r_1}$ to each block-diagonal matrix; see Remark~\ref{rem:classical-vs-quantum-trace}), gives an additive-combinatorial reading of the GB-code $4$-cycle formula via additive energies (Lemma~\ref{lem:parseval-energy}, Proposition~\ref{prop:4cycles}), and a joint Sidon condition characterising girth-$6$ GB codes (Corollary~\ref{cor:sidon}).

\subsection{DFT block-diagonalisation}

Consider a lifted product code over $\mathbb{Z}_\ell$ (Definition~\ref{def:liftedproduct}) with base matrices $H_1 \in \mathbb{F}_2[\mathbb{Z}_\ell]^{r_1 \times n_1}$ and $H_2 \in \mathbb{F}_2[\mathbb{Z}_\ell]^{r_2 \times n_2}$, whose entries lie in the group ring $\mathbb{F}_2[\mathbb{Z}_\ell] \cong \mathbb{F}_2[x]/(x^\ell - 1)$.
We specialise to the QC-GHP case $r_2 = n_2 = 1$, so $H_2$ is the $1 \times 1$ matrix with entry $b(x) \in \mathbb{F}_2[x]/(x^\ell - 1)$. This covers the generalised bicycle ansatz (Example~\ref{ex:gb}) and the QC-GHP family (Example~\ref{ex:ghp}), including the running $[\![48, 6, 8]\!]$ and $[\![882, 24, d]\!]$ codes; in the higher-rank case $r_2, n_2 > 1$ the matrix~\eqref{eq:Mj} becomes the commuting Kronecker sum $M_j = A(\omega^j)A(\omega^j)^* \otimes I_{n_2} + I_{r_1} \otimes B(\omega^j)^* B(\omega^j)$, where $B(\omega^j)$ is the $r_2 \times n_2$ Fourier evaluation of $H_2$, so its eigenvalues are the pairwise sums of those of the two factors and the trace analysis of \S\ref{sec:spectral-cycles} carries over (see Remark~\ref{rem:classical-vs-quantum-trace}).
The $X$-parity check matrix is then
\[
    H_X = \bigl[\,\widetilde{H}_1 \;\big|\; I_{r_1} \otimes \widetilde{B}^{\mkern1mu T}\,\bigr],
\]
where $\widetilde{H}_1$ is the $(r_1 \ell) \times (n_1 \ell)$ matrix obtained by replacing each group-ring entry of $H_1$ with the corresponding $\ell \times \ell$ circulant matrix (and each zero entry with $0_{\ell \times \ell}$), and $\widetilde{B}$ is the $\ell \times \ell$ circulant associated with $b(x)$.

Let $\omega = e^{2\pi i/\ell}$.
Every $\ell \times \ell$ circulant $C$ with first column $(c_0, c_1, \ldots, c_{\ell-1})^T$ is diagonalised by the DFT matrix $F_\ell$: $F_\ell\, C\, F_\ell^* = \operatorname{diag}(\hat{c}_0, \ldots, \hat{c}_{\ell-1})$, where $\hat{c}_j = c(\omega^j) = \sum_{s=0}^{\ell-1} c_s\, \omega^{js}$.
Since $H_X$ has block structure $[\widetilde{H}_1 \mid I_{r_1} \otimes \widetilde{B}^T]$ with $(r_1 \ell)$ rows and $(n_1 + r_1)\ell$ columns, conjugate on the left by $U_L = I_{r_1} \otimes F_\ell$ and on the right by $U_R = \operatorname{diag}(I_{n_1} \otimes F_\ell^*,\; I_{r_1} \otimes F_\ell^*)$.
For the first block, $U_L\, \widetilde{H}_1\, (I_{n_1} \otimes F_\ell^*)$ diagonalises each circulant sub-block independently.
For the second block, since $b(x)$ has real (binary) coefficients, $\widetilde{B}^T$ corresponds to $b(x^{-1})$ and satisfies $F_\ell\, \widetilde{B}^T\, F_\ell^* = \operatorname{diag}(\overline{b(\omega^0)}, \ldots, \overline{b(\omega^{\ell-1})})$, since $b(\omega^{-j}) = \overline{b(\omega^j)}$ for real coefficients.
Grouping indices by frequency $j$ rather than block index $p$ then gives the block-diagonal decomposition
\begin{equation}\label{eq:block-diag}
    \widehat{H}_X = \bigoplus_{j=0}^{\ell-1} \widehat{H}_X^{(j)},
    \qquad
    \widehat{H}_X^{(j)} = \bigl[\, A(\omega^j) \;\big|\; \overline{b(\omega^j)}\, I_{r_1} \,\bigr] \in \mathbb{C}^{r_1 \times (n_1 + r_1)},
\end{equation}
where $A(\omega^j)$ is the $r_1 \times n_1$ matrix obtained by evaluating each group-ring entry of $H_1$ at $\omega^j$; see Figure~\ref{fig:dft-blockdiag}.

\begin{figure}[htp]
\centering
\begin{tikzpicture}[thick]
\draw[ugentblue, fill=ugentblue!10] (-5.5, 0.0) rectangle (-3.6, 1.6);
\draw[red, fill=red!10] (-3.6, 0.0) rectangle (-2.6, 1.6);
\node[font=\scriptsize] at (-4.55, 0.8) {$\widetilde{H}_1$};
\node[font=\scriptsize] at (-3.1, 0.8) {$I \otimes \widetilde{B}^T$};
\node[font=\scriptsize, below] at (-4.05, -0.1) {$r_1 \ell \times (n_1 + r_1)\ell$};
\node[above] at (-4.05, 1.6) {$H_X$};
\draw[->, very thick, black!60] (-2.3, 0.8) -- (-0.9, 0.8);
\node[font=\scriptsize, align=center, above] at (-1.6, 0.9) {$U_L\,(\cdot)\,U_R$};
\node[font=\scriptsize, align=center, below] at (-1.6, 0.7) {block-diag basis};
\foreach \j/\x in {0/-0.4, 1/0.4, 2/1.2, 3/2.0, 4/2.8} {
  \ifnum\j<4
    \draw[ugentblue, fill=ugentblue!10] (\x, 1.2 - \j*0.5) rectangle (\x+0.5, 1.6 - \j*0.5);
    \draw[red, fill=red!10] (\x+0.5, 1.2 - \j*0.5) rectangle (\x+0.7, 1.6 - \j*0.5);
  \fi
}
\node at (1.7, 0.3) {$\ddots$};
\foreach \j/\y in {4/-0.7} {
  \draw[ugentblue, fill=ugentblue!10] (2.0, \y) rectangle (2.5, \y+0.4);
  \draw[red, fill=red!10] (2.5, \y) rectangle (2.7, \y+0.4);
}
\node[font=\scriptsize, right] at (0.2, 1.4) {$\widehat{H}_X^{(0)} = [A(1) \mid b(1) I_{r_1}]$};
\node[font=\scriptsize, right] at (1.0, 0.9) {$\widehat{H}_X^{(1)}$};
\node[font=\scriptsize, right] at (1.8, 0.4) {$\widehat{H}_X^{(2)}$};
\node[font=\scriptsize, right] at (2.7, -0.5) {$\widehat{H}_X^{(\ell-1)}$};
\node[above] at (1.5, 1.7) {$\widehat{H}_X = \bigoplus_{j=0}^{\ell-1} \widehat{H}_X^{(j)}$};
\node[font=\scriptsize, below] at (1.5, -1.0) {each block: $r_1 \times (n_1 + r_1)$};
\end{tikzpicture}
\caption[DFT block-diagonalisation of $H_X$]{The discrete Fourier transform takes the large parity-check matrix $H_X = [\widetilde{H}_1 \mid I_{r_1} \otimes \widetilde{B}^T]$ on the left (size $r_1 \ell \times (n_1 + r_1) \ell$, with the first block in blue and the second in red) to a direct sum of $\ell$ much smaller blocks $\widehat{H}_X^{(j)}$ on the right (each of size $r_1 \times (n_1 + r_1)$, indexed by the Fourier frequency $j \in \{0, 1, \dots, \ell - 1\}$). The first block of every $\widehat{H}_X^{(j)}$ is the Fourier evaluation $A(\omega^j)$ of the protograph data and the second block is the scalar $\overline{b(\omega^j)} I_{r_1}$. Theorem~\ref{prop:spectral-trace} below uses this decomposition to reduce trace-of-power computations on $H_X H_X^T$ to a sum of traces on the smaller matrices $M_j = A(\omega^j) A(\omega^j)^* + |b(\omega^j)|^2 I_{r_1}$.}
\label{fig:dft-blockdiag}
\end{figure}
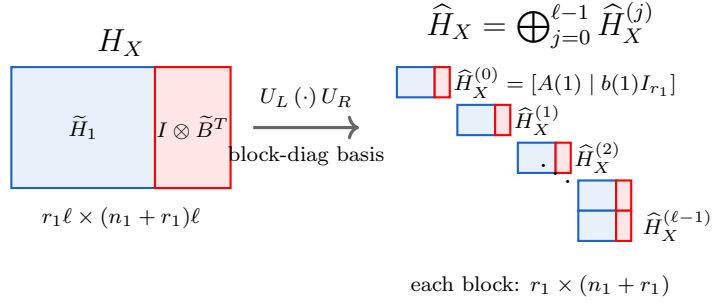

\begin{thm}[Spectral trace formula for lifted product codes]
\label{prop:spectral-trace}
Let $H_X$ be the parity check matrix of a lifted product code over $\mathbb{Z}_\ell$ as above.
For each $j = 0, \ldots, \ell - 1$, define the $r_1 \times r_1$ positive semidefinite Hermitian matrix
\begin{equation}\label{eq:Mj}
    M_j \;=\; A(\omega^j)\, A(\omega^j)^* \;+\; |b(\omega^j)|^2\, I_{r_1}.
\end{equation}
Then, by Lemma~\ref{lem:bipartite-trace}, the number of closed walks of length $2k$ in the Tanner graph of $H_X$ starting from check nodes is
\begin{equation}\label{eq:trace-formula}
    \operatorname{tr}\bigl((H_X H_X^T)^k\bigr)
    \;=\; \sum_{j=0}^{\ell-1} \operatorname{tr}(M_j^k)
    \;=\; \sum_{j=0}^{\ell-1} \sum_{i=1}^{r_1} \lambda_i(M_j)^k,
\end{equation}
where $\lambda_1(M_j) \leq \cdots \leq \lambda_{r_1}(M_j)$ are the eigenvalues of $M_j$. Positive semidefiniteness of $M_j$ ensures that all eigenvalues are non-negative real, so the right-hand side of~\eqref{eq:trace-formula} is a sum of non-negative real-valued powers and the cycle-count interpretation is consistent with the left-hand side being a real count.
\end{thm}

\begin{proof}
Treat $H_X$ as a real $0$-$1$ matrix, so $H_X^T = H_X^*$.
With $U_L, U_R$ as above, the conjugated matrix $\widehat{H}_X = U_L\, H_X\, U_R$ has the block-diagonal form~\eqref{eq:block-diag} after the frequency-first permutation $P$.
As $U_L, U_R$ are unitary and $P$ is a permutation matrix (hence unitary), these similarities preserve traces of matrix powers, so
\[
    \operatorname{tr}\bigl((H_X H_X^T)^k\bigr)
    = \operatorname{tr}\bigl((U_L\, H_X H_X^T\, U_L^*)^k\bigr)
    = \operatorname{tr}\bigl((\widehat{H}_X \widehat{H}_X^*)^k\bigr)
    = \sum_{j=0}^{\ell-1} \operatorname{tr}\Bigl(\bigl(\widehat{H}_X^{(j)} (\widehat{H}_X^{(j)})^*\bigr)^k\Bigr),
\]
the last equality using the block-diagonal structure (see the spectral background \S\ref{sec:spectral-background} for matrix powers as walk counts and the trace-of-powers and DFT identities). Finally,
\begin{align*}
    \widehat{H}_X^{(j)} (\widehat{H}_X^{(j)})^*
    &= \bigl[\, A(\omega^j) \;\big|\; \overline{b(\omega^j)}\, I_{r_1} \,\bigr]
    \begin{bmatrix} A(\omega^j)^* \\ b(\omega^j)\, I_{r_1} \end{bmatrix}
   \\
   &= A(\omega^j) A(\omega^j)^* + |b(\omega^j)|^2 I_{r_1}
   \\
   &= M_j,
\end{align*}
which completes the proof.
\end{proof}

\begin{rem}\label{rem:classical-vs-quantum-trace}
Theorem~\ref{prop:spectral-trace} reduces the computation of traces of $(r_1 \ell) \times (r_1 \ell)$ matrix powers to a sum of traces of $r_1 \times r_1$ matrix powers.
For the $[\![882,24,d]\!]$ code, this is a reduction from $441 \times 441$ matrices to $63$ copies of $7 \times 7$ matrices.

For a single classical $J \times r$ protograph parity-check matrix lifted by a factor $L$, the eigenvalue reduction $H^T H \to \{H(\omega^j) H(\omega^j)^T\}_{j=0}^{L-1}$ that underlies Theorem~\ref{prop:spectral-trace} is due to Smarandache and Flanagan~\cite{SF2009}, and the explicit cycle-count formulas $N_{2k} = N_{2k}(\{\lambda_i\})$ for $k \leq 7$ are derived in Smarandache and Mitchell~\cite{SM2026} for any biregular Tanner graph. Karimi and Banihashemi~\cite{karimi2012} give a parallel DFT recipe via the directed edge matrix; the two approaches are equivalent up to a change of basis. What this section adds to that classical literature is the specialisation to the \emph{quantum} two-block structure $H_X = [\widetilde{H}_1 \mid I_{r_1} \otimes \widetilde{B}^T]$: the second block $I_{r_1} \otimes \widetilde{B}^T$ contributes only a scalar shift $|b(\omega^j)|^2 I_{r_1}$ to each $M_j$, so the spectrum of $M_j$ is that of $A(\omega^j) A(\omega^j)^*$ from the first block, shifted uniformly by the second. This scalar-shift form would not hold for a general two-block concatenation with independent column structure.
\end{rem}

\subsection{Applications}\label{subsec:applications}

In this section we provide some applications of the proposed spectral framework. In particular we illustrate how to use the spectral trace formula of Theorem~\ref{prop:spectral-trace} in two directions: the scalar reduction for generalised bicycle codes, and the genuinely matrix-valued QC-GHP case worked out on the running $[\![882,24,d]\!]$ code.

\subsubsection{Specialisation to generalised bicycle codes}\label{subsubsec:gb-specialisation}

\begin{cor}[Scalar trace formula for GB codes]\label{cor:trace-gb}
When $r_1=n_1=1$ (generalised bicycle codes), $A(\omega^j) = a(\omega^j)$ is a scalar, so $M_j$ reduces to the scalar $|a(\omega^j)|^2 + |b(\omega^j)|^2$, and the expression~\eqref{eq:trace-formula} becomes
\begin{equation}\label{eq:trace-gb}
    \operatorname{tr}\bigl((H_X H_X^T)^k\bigr) = \sum_{j=0}^{\ell-1} \bigl(|a(\omega^j)|^2 + |b(\omega^j)|^2\bigr)^k.
\end{equation}
\end{cor}
\begin{proof}
Use Theorem~\ref{prop:spectral-trace} with $r_1=n_1=1$, where $A(\omega^j)=a(\omega^j)$ is a scalar.
\end{proof}

\subsubsection{Application to the $[\![882,24,d]\!]$ QC-GHP code}\label{subsubsec:ghp-882}

This subsubsection applies the spectral trace formula (Theorem~\ref{prop:spectral-trace}) to the genuinely matrix-valued QC-GHP case, computing the $4$-cycle count of the running $[\![882,24,d]\!]$ code.
For QC-GHP codes $A(\omega^j)$ is genuinely $r_1 \times r_1$, so~\eqref{eq:Mj} no longer collapses to a scalar shift.
We illustrate on the $[\![882,24,d]\!]$ code of Example~\ref{ex:ghp}, where a second layer of cyclic symmetry inside $H_1$ yields the eigenvalues of each $M_j$ in closed form.

\begin{ex}[{Spectral cycle counting for the $[\![882,24,d]\!]$ code}]\label{ex:spectral-882}
With $\ell = 63$, $r_1 = n_1 = 7$, and $b(x) = 1 + x + x^6$ as in Example~\ref{ex:ghp}, the protograph matrix
\[
    H_1 \;=\; x^{27}\, I_7 \;+\; S \;+\; x^{54}\, S^2,
\]
is itself a polynomial in the $7 \times 7$ cyclic shift $S$ defined by $(S)_{i,k} = 1$ iff $k \equiv i - 1 \pmod{7}$ (the three diagonals of $H_1$ in Example~\ref{ex:ghp} correspond to the three terms of this polynomial).
Evaluating each group-ring entry at $\omega^j$, where $\omega = e^{2\pi i / 63}$, gives the $7 \times 7$ matrix
\[
    A(\omega^j) \;=\; \omega^{27 j}\, I_7 \;+\; S \;+\; \omega^{54 j}\, S^2,
\]
which is itself a polynomial in $S$, so it is simultaneously diagonalised with $S$ in the basis of $\mathbb{Z}_7$-Fourier vectors.
The eigenvalues of $S$ are $\eta^{-m}$ for $m = 0, \ldots, 6$ with $\eta = e^{2\pi i / 7} = \omega^9$, so the eigenvalues of $A(\omega^j)$ are
\begin{equation}\label{eq:882-mu}
    \mu_{j,m} \;=\; \omega^{27 j} \;+\; \omega^{-9 m} \;+\; \omega^{54 j - 18 m}, \qquad m = 0, \ldots, 6.
\end{equation}
Since $A(\omega^j)$ commutes with $S$ and is normal, $A(\omega^j) A(\omega^j)^*$ has eigenvalues $|\mu_{j,m}|^2$, and so by Theorem~\ref{prop:spectral-trace} the matrix $M_j$ has eigenvalues $|\mu_{j,m}|^2 + |b(\omega^j)|^2$.
The spectral trace formula~\eqref{eq:trace-formula} therefore reduces to the explicit double sum
\begin{equation}\label{eq:882-trace}
    \operatorname{tr}\bigl((H_X H_X^T)^k\bigr)
    \;=\; \sum_{j=0}^{62}\, \sum_{m=0}^{6}\, \bigl(|\mu_{j,m}|^2 + |b(\omega^j)|^2\bigr)^k,
\end{equation}
turning a $441 \times 441$ trace computation into $441$ scalar evaluations of~\eqref{eq:882-mu} followed by a $k$-th power.

As a consistency check at $k = 1$: each $A(\omega^j)$ has Frobenius norm squared $\|A(\omega^j)\|_F^2 = 21$ (it has $21$ unimodular entries), so $\sum_m |\mu_{j,m}|^2 = 21$ for every $j$, giving $\sum_{j,m} |\mu_{j,m}|^2 = 63 \cdot 21 = 1323$.
By Parseval, $\sum_j |b(\omega^j)|^2 = \ell \cdot \operatorname{wt}(b) = 189$, so $\sum_{j,m} |b(\omega^j)|^2 = 7 \cdot 189 = 1323$.
The two contributions sum to $\operatorname{tr}(H_X H_X^T) = 2646 = 6 \cdot 441$, as expected from the $(3,6)$-regularity of the Tanner graph.

At $k = 2$, Lemma~\ref{lem:4cycle-trace} below with $\sum_i d_{c,i}^2 = 441 \cdot 36 = 15876$ and $\sum_v d_v(d_v - 1) = 882 \cdot 6 = 5292$ gives
\[
    C_4 \;=\; \tfrac{1}{4}\bigl(\operatorname{tr}\bigl((H_X H_X^T)^2\bigr) \;-\; 21168\bigr).
\]
\end{ex}

\begin{cor}[{$4$-cycle count of the $[\![882,24,d]\!]$ code}]\label{cor:882-4cycle}
The Tanner graph of the $[\![882,24,d]\!]$ code has girth~$6$ \cite{panteleev2021degenerate}, equivalently $C_4 = 0$, so the protograph parameters $(27, 54)$ together with $b(x) = 1 + x + x^6$ satisfy the closed-form identity
\begin{equation}\label{eq:882-girth-test}
    \sum_{j=0}^{62}\sum_{m=0}^{6}\bigl(|\mu_{j,m}|^2 + |b(\omega^j)|^2\bigr)^2 \;=\; 21168.
\end{equation}
\end{cor}
\begin{proof}
Substitute the exponents $(27, 54)$, the modulus $\ell = 63$, and $b(x) = 1 + x + x^6$ into the double sum~\eqref{eq:882-trace} at $k = 2$, whose value is $\operatorname{tr}\bigl((H_X H_X^T)^2\bigr)$. Since the girth is~$6$ we have $C_4 = 0$, so the $k = 2$ reduction $C_4 = \tfrac{1}{4}\bigl(\operatorname{tr}\bigl((H_X H_X^T)^2\bigr) - 21168\bigr)$ forces $\operatorname{tr}\bigl((H_X H_X^T)^2\bigr) = 21168$, which is the left-hand side of~\eqref{eq:882-girth-test}.
\end{proof}
The identity~\eqref{eq:882-girth-test} is a sharp $4$-cycle test: any modification of the exponents $(27, 54)$ or of $b(x)$ that violates it introduces $4$-cycles into the Tanner graph, and every nonnegative deviation from $21168$ equals $4 C_4$.

\subsection{Cycle counts via additive energy}

Next we connect the formula~\eqref{eq:trace-gb} to additive combinatorics over $\mathbb{Z}_\ell$, so that $C_4 = 0$ becomes an arithmetic condition on the supports $S_a, S_b \subset \mathbb{Z}_\ell$ rather than a closed-walk count.
Write $S_a = \operatorname{supp}(a) \subset \mathbb{Z}_\ell$ and $S_b = \operatorname{supp}(b)$ for the supports of $a(x)$ and $b(x)$.
\begin{defn}[Additive energy and cross-energy]\label{def:additive-energy}
    The \emph{additive energy of a set} $S \subset \mathbb{Z}_\ell$ is
    $$E(S) = \bigl|\bigl\{(i,k,p,q) \in S^4 : i+ p  \equiv k+q \pmod{\ell}\bigr\}\bigr| = \sum_{d \in \mathbb{Z}_\ell} r_d(S)^2,
    $$
    where $r_d(S) = |\{(x,y) \in S^2 : x - y \equiv d\}|$.
    The \emph{cross-energy} of the sets $S_a$ and $S_b\subset \mathbb{Z}_\ell$ is
\[
    E(S_a, S_b) = \sum_{d \in \mathbb{Z}_\ell} r_d(S_a) \cdot r_d(S_b).
\]
\end{defn}

\begin{ex}\label{ex:additive-energy-small}
For $S = \{0, 1, 3\} \subset \mathbb{Z}_7$, the difference counts are $r_0(S) = 3$ and $r_{\pm 1}(S) = r_{\pm 2}(S) = r_{\pm 3}(S) = 1$, giving $E(S) = 9 + 6 = 15$ (Figure~\ref{fig:additive-energy}). Proposition~\ref{prop:4cycles} below expresses the $4$-cycle count of a generalised bicycle code through these quantities.
\end{ex}

\begin{figure}[ht]
\centering
\begin{subfigure}[t]{0.40\textwidth}
\centering
\begin{tikzpicture}[thick, scale=1.2,
  inSet/.style={circle, draw=black, fill=ugentblue, minimum size=8pt, inner sep=0pt},
  outSet/.style={circle, draw=black, fill=white, minimum size=8pt, inner sep=0pt}]
\foreach \k in {0,1,2,3,4,5,6} {
  \pgfmathsetmacro{\ang}{90 - 360*\k/7}
  \pgfmathsetmacro{\inS}{(\k==0 || \k==1 || \k==3) ? 1 : 0}
  \ifnum\inS=1
    \node[inSet] (z\k) at (\ang:1.1) {};
  \else
    \node[outSet] (z\k) at (\ang:1.1) {};
  \fi
  \node[font=\scriptsize] at ($(\ang:1.45)$) {$\k$};
}
\end{tikzpicture}
\caption{$S = \{0, 1, 3\} \subset \mathbb{Z}_7$ (blue).}
\label{fig:additive-energy-set}
\end{subfigure}
\hfill
\begin{subfigure}[t]{0.55\textwidth}
\centering
\begin{tikzpicture}[thick, scale=0.85]
\draw[->] (0, 0) -- (8, 0) node[right, font=\scriptsize] {$d \in \mathbb{Z}_7$};
\draw[->] (0, 0) -- (0, 3.6) node[above, font=\scriptsize] {$r_d(S)$};
\foreach \d/\h in {0/3, 1/1, 2/1, 3/1, 4/1, 5/1, 6/1} {
  \fill[ugentblue] (\d + 0.4, 0) rectangle (\d + 0.9, \h);
  \node[font=\scriptsize] at (\d + 0.65, -0.3) {$\d$};
  \node[font=\scriptsize] at (\d + 0.65, \h + 0.25) {$\h$};
}
\node[font=\scriptsize, align=left] at (5.6, 3.0) {$E(S) = \sum_d r_d^2$\\$= 9 + 6 = 15$};
\end{tikzpicture}
\caption{Difference-multiplicity profile $r_d(S)$.}
\label{fig:additive-energy-rd}
\end{subfigure}
\caption[Additive energy on $\mathbb{Z}_7$]{Additive energy for the small example $S = \{0, 1, 3\} \subset \mathbb{Z}_7$. \textbf{(a)} The set $S$ is shown as the three blue points on the cyclic group $\mathbb{Z}_7$ (white points are non-members). \textbf{(b)} The histogram $r_d(S) = |\{(x, y) \in S^2 : x - y \equiv d\}|$ takes value $3$ at $d = 0$ (every pair $(s, s)$ contributes) and value $1$ at each $d \in \{1, 2, 3, 4, 5, 6\}$ (the six non-zero pairwise differences of $S$ are pairwise distinct in $\mathbb{Z}_7$, which makes $S$ a \emph{Sidon set}). Summing $r_d(S)^2$ gives $E(S) = 3^2 + 6 \cdot 1^2 = 15$ quadruples $(i, k, p, q) \in S^4$ with $i + p \equiv k + q \pmod 7$.}
\label{fig:additive-energy}
\end{figure}
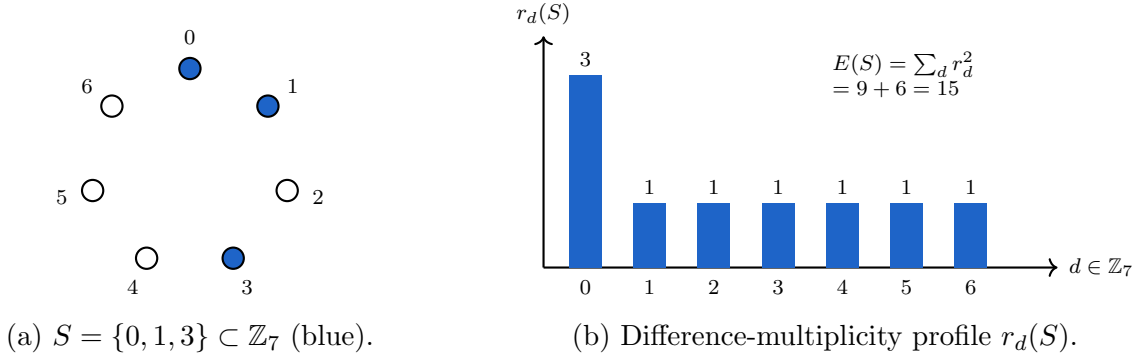
\begin{rem}\label{rem:cross-energy-convention}
Since $r_{-d}(S) = r_d(S)$ for any $S \subset \mathbb{Z}_\ell$ (by the substitution $(x,y) \mapsto (y,x)$), the definition $E(S_a, S_b) = \sum_d r_d(S_a)\, r_d(S_b)$, which naturally counts quadruples with $i - k \equiv p - q$ (i.e.\ $i - k - p + q \equiv 0$), also equals the count of quadruples with $i - k + p - q \equiv 0$.
This is the form that arises from the Fourier expansion in Lemma~\ref{lem:parseval-energy}(ii).
\end{rem}

\begin{lem}\label{lem:parseval-energy}
Let $S \subset \mathbb{Z}_\ell$ and $a(x) = \sum_{s \in S} x^s$.
Then:
\begin{enumerate}
    \item[\textnormal{(i)}] $\displaystyle\sum_{j=0}^{\ell-1} |a(\omega^j)|^4 = \ell \cdot E(S)$.
    \item[\textnormal{(ii)}] For $T \subset \mathbb{Z}_\ell$ with $b(x) = \sum_{t \in T} x^t$, $\displaystyle\sum_{j=0}^{\ell-1} |a(\omega^j)|^2\,|b(\omega^j)|^2 = \ell \cdot E(S, T)$.
\end{enumerate}
\end{lem}

\begin{proof}
For (i), write $|a(\omega^j)|^2 = \sum_{i,k \in S} \omega^{j(i-k)}$, so that
\[
    |a(\omega^j)|^4 = \Bigl(\sum_{i,k \in S} \omega^{j(i-k)}\Bigr)\Bigl(\sum_{p,q \in S} \omega^{j(p-q)}\Bigr) = \sum_{\substack{i,k,p,q \in S}} \omega^{j(i - k + p - q)}.
\]
Summing over $j$ and using the orthogonality relation $\sum_{j=0}^{\ell-1} \omega^{jm} = \ell\,\mathbf{1}_{m \equiv 0}$ gives
\[
    \sum_{j=0}^{\ell-1} |a(\omega^j)|^4 = \ell \cdot \bigl|\{(i,k,p,q) \in S^4 : i - k + p - q \equiv 0\}\bigr| = \ell \cdot E(S).
\]
For (ii), the same argument gives $$\sum_j |a(\omega^j)|^2 |b(\omega^j)|^2 = \ell \cdot |\{(i,k,p,q) : i,k \in S,\; p,q \in T,\; i - k + p - q \equiv 0\}|,$$ which equals $\ell \cdot E(S, T)$ by Remark~\ref{rem:cross-energy-convention}.
\end{proof}

\begin{rem}\label{rem:energy-computation}
Lemma~\ref{lem:parseval-energy} gives two interchangeable ways to compute $E(S)$ in practice: combinatorially as $\sum_d r_d(S)^2$ by enumerating differences (cost $O(|S|^2 + \ell)$), or analytically as $\frac{1}{\ell}\sum_j |a(\omega^j)|^4$ via an FFT (cost $O(\ell \log \ell)$).
The same alternative applies to $E(S_a, S_b)$ via part (ii).
\end{rem}

We also use the standard $4$-cycle count in a bipartite graph via the trace of the squared adjacency matrix.

\begin{lem}[see e.g.\ {\cite{halfordchugg2006}}]\label{lem:4cycle-trace}
In a bipartite graph with biadjacency matrix $H$ (check nodes as rows, variable nodes as columns), the number of 4-cycles is
\[
    C_4 = \frac{1}{4}\Bigl(\operatorname{tr}\bigl((H H^T)^2\bigr) - \sum_{i} d_{c,i}^2 - \sum_v d_v(d_v - 1)\Bigr),
\]
where $d_{c,i}$ is the degree of check node $i$ and $d_v$ is the degree of variable node $v$.
\end{lem}

\begin{proof}
Write $A = HH^T$.
The diagonal entry $A_{ii} = d_{c,i}$ and the off-diagonal entry $A_{ij}$ for $i \neq j$ counts the number of common variable-node neighbours of check nodes $i$ and $j$; call this $\lambda_{ij}$.
A 4-cycle is determined by a pair of check nodes $\{i, j\}$ and a pair of their common variable neighbours, so $C_4 = \sum_{i < j} \binom{\lambda_{ij}}{2}$.
Now $\operatorname{tr}(A^2) = \sum_{i,j} A_{ij}^2 = \sum_i d_{c,i}^2 + \sum_{i \neq j} \lambda_{ij}^2$, so $\sum_{i \neq j} \lambda_{ij}^2 = \operatorname{tr}(A^2) - \sum_i d_{c,i}^2$.
Similarly, $\sum_{i \neq j} \lambda_{ij} = \sum_{i,j} A_{ij} - \sum_i A_{ii} = \sum_v d_v^2 - \sum_v d_v = \sum_v d_v(d_v - 1)$, where the identity $\sum_{i,j} A_{ij} = \sum_v d_v^2$ holds because each variable node $v$ contributes~$1$ to $A_{ij}$ for each of the $\binom{d_v}{1}^2 = d_v^2$ pairs $(i,j)$ of check nodes it is adjacent to.
Therefore
\[
    C_4 = \sum_{i < j} \tfrac{1}{2}(\lambda_{ij}^2 - \lambda_{ij})
    = \tfrac{1}{4}\bigl(\operatorname{tr}(A^2) - \textstyle\sum_i d_{c,i}^2\bigr) - \tfrac{1}{4} \textstyle\sum_v d_v(d_v - 1),
\]
using $\sum_{i<j} = \frac{1}{2}\sum_{i \neq j}$.
\end{proof}

\begin{prop}[4-cycle formula for GB codes]
\label{prop:4cycles}
Let $H_X = [\widetilde{A} \mid \widetilde{B}]$ be the parity check matrix of a generalised bicycle code over $\mathbb{Z}_\ell$ with $\operatorname{wt}(a) = w_a$ and $\operatorname{wt}(b) = w_b$.
The Tanner graph of $H_X$ has $\ell$ check nodes, each of degree $d_c = w_a + w_b$, and $2\ell$ variable nodes: $\ell$ of degree $w_a$ (from the $\widetilde{A}$ block) and $\ell$ of degree $w_b$ (from the $\widetilde{B}$ block).
The number of 4-cycles is
\begin{equation}\label{eq:4cycles}
    C_4 = \frac{\ell}{4}\bigl(E(S_a) + 2\,E(S_a, S_b) + E(S_b)\bigr)
    - \frac{\ell \bigl((w_a + w_b)^2 + w_a(w_a - 1) + w_b(w_b - 1)\bigr)}{4}.
\end{equation}
\end{prop}

\begin{proof}
We use~\eqref{eq:trace-gb} with $k = 2$:
\[
    \operatorname{tr}\bigl((H_X H_X^T)^2\bigr) = \sum_{j=0}^{\ell-1} \bigl(|a(\omega^j)|^2 + |b(\omega^j)|^2\bigr)^2.
\]
Expanding the square and applying Lemma~\ref{lem:parseval-energy}:
\begin{align*}
    \operatorname{tr}\bigl((H_X H_X^T)^2\bigr)
    &= \sum_{j} |a(\omega^j)|^4 + 2\sum_j |a(\omega^j)|^2|b(\omega^j)|^2 + \sum_j |b(\omega^j)|^4 \\
    &= \ell\, E(S_a) + 2\ell\, E(S_a, S_b) + \ell\, E(S_b).
\end{align*}
Substituting into Lemma~\ref{lem:4cycle-trace} with $m = \ell$, $\sum_i d_{c,i}^2 = \ell(w_a + w_b)^2$, and $\sum_v d_v(d_v - 1) = \ell\, w_a(w_a - 1) + \ell\, w_b(w_b - 1)$ gives the desired result~\eqref{eq:4cycles}.
\end{proof}

\begin{rem}
Proposition~\ref{prop:4cycles} shows that the girth of a GB code Tanner graph is at least~6 if and only if
\[
    E(S_a) + 2\,E(S_a, S_b) + E(S_b) = (w_a + w_b)^2 + w_a(w_a - 1) + w_b(w_b - 1).
\]
The right-hand side is the contribution from ``trivial'' 4-tuples (those where the two pairs of indices coincide).
The following corollary reformulates this as a condition from additive combinatorics.
\end{rem}

\begin{cor}[Joint Sidon condition for girth-6 GB codes]
\label{cor:sidon}
A generalised bicycle code over $\mathbb{Z}_\ell$ with polynomials $a(x), b(x)$ has girth at least~$6$ if and only if
\begin{equation}\label{eq:sidon}
    r_d(S_a) + r_d(S_b) \leq 1 \quad \text{for all } d \not\equiv 0 \pmod{\ell},
\end{equation}
where $r_d(S) = |\{(x,y) \in S^2 : x - y \equiv d \pmod{\ell}\}|$.
In other words, the difference multisets of $S_a$ and $S_b$, restricted to nonzero differences, must be disjoint and repeat-free when combined.
\end{cor}

\begin{proof}
By Proposition~\ref{prop:4cycles}, $C_4 = 0$ if and only if $E(S_a) + 2\,E(S_a, S_b) + E(S_b) = (w_a + w_b)^2 + w_a(w_a - 1) + w_b(w_b - 1)$.
Since $E(S_a) + 2\,E(S_a, S_b) + E(S_b) = \sum_d (r_d(S_a) + r_d(S_b))^2$, the left-hand side decomposes as
\[
    (r_0(S_a) + r_0(S_b))^2 + \sum_{d \neq 0} (r_d(S_a) + r_d(S_b))^2 = (w_a + w_b)^2 + \sum_{d \neq 0} (r_d(S_a) + r_d(S_b))^2,
\]
using $r_0(S) = |S|$.
The right-hand side equals $(w_a + w_b)^2 + w_a(w_a - 1) + w_b(w_b - 1)$.
Since $\sum_{d \neq 0} r_d(S_a) = w_a^2 - w_a = w_a(w_a - 1)$ and likewise for $S_b$, the condition $C_4 = 0$ reduces to
\[
    \sum_{d \neq 0} (r_d(S_a) + r_d(S_b))^2 = \sum_{d \neq 0} (r_d(S_a) + r_d(S_b)).
\]
For non-negative integers $x$, $x^2 \geq x$ with equality iff $x \in \{0, 1\}$, so $$\sum_{d \neq 0}\bigl((r_d(S_a)+r_d(S_b))^2 - (r_d(S_a)+r_d(S_b))\bigr) = 0$$ with each summand non-negative forces every summand to vanish, giving~\eqref{eq:sidon}.
\end{proof}

\begin{rem}
When $S_b = \emptyset$ (i.e.\ a single circulant block), condition~\eqref{eq:sidon} reduces to $r_d(S_a) \leq 1$ for all $d \neq 0$, which is the classical definition of a \emph{Sidon set} (or $B_2$ set) in $\mathbb{Z}_\ell$.
For a Sidon set $S \subset \mathbb{Z}_\ell$, counting nonzero differences gives $|S|(|S| - 1) \leq \ell - 1$, so $|S| \leq \sqrt{\ell} + O(1)$ (cf.\ the Erd\H{o}s--Tur\'an bound for Sidon sets of integers~\cite{erdosturan1941}).
For the joint condition~\eqref{eq:sidon}, the combined difference set $\{x - y : (x,y) \in S_a^2 \cup S_b^2,\, x \neq y\}$ must have no repeated elements, so
\[
    w_a(w_a - 1) + w_b(w_b - 1) \leq \ell - 1,
\]
giving an upper bound on the total weight $w_a + w_b$ achievable at girth~6 for a given $\ell$.
\end{rem}

\begin{rem}[Towards $6$-cycle counts via additive $3$-energies]\label{rem:6cycles}
The same DFT identity extends to closed walks of length~$6$: by Theorem~\ref{prop:spectral-trace} with $k = 3$,
\[
    \operatorname{tr}\bigl((H_X H_X^T)^3\bigr) \;=\; \sum_{j=0}^{\ell-1}\bigl(|a(\omega^j)|^2 + |b(\omega^j)|^2\bigr)^3.
\]
Expanding the cube and applying the same Parseval calculation as in Lemma~\ref{lem:parseval-energy} (using the Parseval identity~\eqref{eq:parseval} of the spectral background \S\ref{sec:spectral-background}) expresses each term via the \emph{additive $3$-energy}
\[
    T_3(S) \;=\; \bigl|\bigl\{(i_1,k_1,i_2,k_2,i_3,k_3) \in S^6 \;:\; (i_1-k_1) + (i_2-k_2) + (i_3-k_3) \equiv 0 \pmod{\ell}\bigr\}\bigr|,
\]
together with mixed cross-versions for $S_a$ and $S_b$, via the identity $\sum_j |a(\omega^j)|^6 = \ell\, T_3(S_a)$ (and analogously for higher mixed terms).
Note that $T_3(S) \neq \sum_d r_d(S)^3$ in general: writing $r_d(S)$ for the number of representations $d = i - k$ with $i,k \in S$, one has $T_3(S) = \sum_{d_1 + d_2 + d_3 \equiv 0 \!\pmod{\ell}} r_{d_1}(S)\,r_{d_2}(S)\,r_{d_3}(S)$, which counts \emph{all} difference-triples summing to $0$ modulo $\ell$, whereas $\sum_d r_d(S)^3$ counts only those with $d_1 = d_2 = d_3$; neither sum is contained in the other.
Translating $\operatorname{tr}\bigl((H_X H_X^T)^3\bigr)$ into a $6$-cycle count requires subtracting closed length-$6$ walks that are not simple cycles; under girth $\geq 6$ (i.e.\ assuming the Sidon condition~\eqref{eq:sidon}) these reduce to backtracking walks whose count is fixed by the local degree statistics of the Tanner graph.
The resulting girth-$8$ criterion is then a higher-order arithmetic condition on the supports, analogous to Corollary~\ref{cor:sidon} but strictly stronger; we do not work out the full bookkeeping here.
\end{rem}

\subsection{Spectral interpretation of the Wang-Dolecek-Wesel reduction in the QC-GHP setting}\label{sec:absorbing-spectral}

The cycle-counting formulas of Section~\ref{sec:spectral-cycles} connect to absorbing sets, the dominant graphical cause of decoder failure, through a classical bijection of Wang, Dolecek and Wesel~\cite{wangdolecekwesel2012}. Recast in graph-theoretic language and combined with Theorem~\ref{prop:spectral-trace}, it yields a closed-form spectral expression for the number of $(3, 3)$ elementary absorbing sets in a QC-GHP code, addressing in the column-weight-$3$ case an enumeration question raised by Morris, Pllaha and Kelley~\cite{morris2024} for QLDPC codes with circulant base matrices.

Throughout, assume the Tanner graph $G$ has constant variable-node degree $\gamma = 3$ and girth at least $6$, which covers the QC-GHP family above when $H_1$ has column weight three and $\operatorname{wt}(b) = 3$. Recall the variable-node graph $G^{VN}$ from Definition~\ref{def:vng}; under girth $\geq 6$ it is simple.

\begin{lem}[Wang-Dolecek-Wesel \cite{wangdolecekwesel2012}, in $G^{VN}$ language]\label{lem:33-triangle}
For a Tanner graph $G$ with $\gamma = 3$ and $g(G) \geq 6$, the assignment $A \mapsto A$ is a bijection between
\begin{enumerate}
    \item[\textnormal{(a)}] elementary $(3, 3)$-absorbing sets of $G$, and
    \item[\textnormal{(b)}] simple $6$-cycles of $G$, equivalently the triangles $K_3$ in $G^{VN}$ realised by three pairwise distinct checks.
\end{enumerate}
A triangle of $G^{VN}$ realised by a single check of degree $\geq 3$ is not absorbing: that check has odd degree in $G_A$, so each of its three variables would have three odd and no even checks, violating the strong-majority condition.
\end{lem}

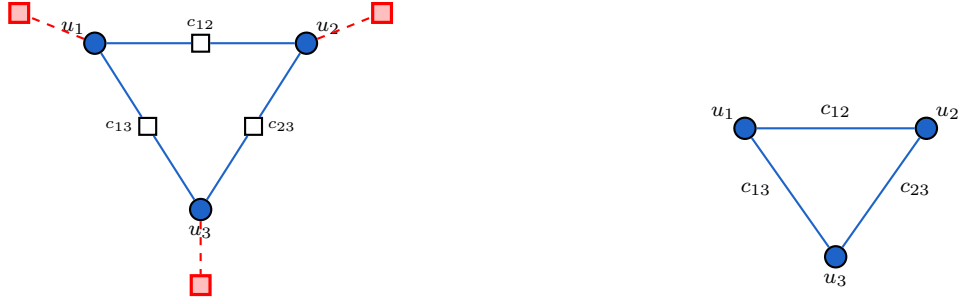
\begin{figure}[ht]
\centering
\begin{subfigure}[t]{0.55\textwidth}
\centering
\begin{tikzpicture}[thick,
  varIn/.style={circle, draw=black, fill=ugentblue, minimum size=8pt, inner sep=0pt},
  chkE/.style={regular polygon, regular polygon sides=4, draw=black, fill=white, minimum size=9pt, inner sep=0pt},
  chkO/.style={regular polygon, regular polygon sides=4, draw=red, fill=red!25, minimum size=10pt, inner sep=0pt, very thick}]
\node[varIn] (u1) at (-1.4,  1.2) {};  \node[above left=-1pt] at (u1) {\scriptsize $u_1$};
\node[varIn] (u2) at ( 1.4,  1.2) {};  \node[above right=-1pt] at (u2) {\scriptsize $u_2$};
\node[varIn] (u3) at ( 0.0, -1.0) {};  \node[below=2pt] at (u3) {\scriptsize $u_3$};
\node[chkE] (c12) at ( 0.0,  1.2) {};  \node[above=1pt] at (c12) {\tiny $c_{12}$};
\node[chkE] (c13) at (-0.7,  0.1) {};  \node[left=1pt] at (c13) {\tiny $c_{13}$};
\node[chkE] (c23) at ( 0.7,  0.1) {};  \node[right=1pt] at (c23) {\tiny $c_{23}$};
\node[chkO] (o1) at (-2.4,  1.6) {};
\node[chkO] (o2) at ( 2.4,  1.6) {};
\node[chkO] (o3) at ( 0.0, -2.0) {};
\draw[ugentblue] (u1) -- (c12) -- (u2);
\draw[ugentblue] (u1) -- (c13) -- (u3);
\draw[ugentblue] (u2) -- (c23) -- (u3);
\draw[red, dashed] (u1) -- (o1);
\draw[red, dashed] (u2) -- (o2);
\draw[red, dashed] (u3) -- (o3);
\end{tikzpicture}
\caption{Elementary $(3, 3)$-absorbing set in the Tanner graph.}
\label{fig:33-tanner}
\end{subfigure}
\hfill
\begin{subfigure}[t]{0.40\textwidth}
\centering
\begin{tikzpicture}[thick,
  varIn/.style={circle, draw=black, fill=ugentblue, minimum size=8pt, inner sep=0pt}]
\node[varIn] (u1) at (-1.2,  0.7) {};  \node[above left=-1pt] at (u1) {\scriptsize $u_1$};
\node[varIn] (u2) at ( 1.2,  0.7) {};  \node[above right=-1pt] at (u2) {\scriptsize $u_2$};
\node[varIn] (u3) at ( 0.0, -1.0) {};  \node[below=2pt] at (u3) {\scriptsize $u_3$};
\draw[ugentblue] (u1) -- (u2);
\draw[ugentblue] (u1) -- (u3);
\draw[ugentblue] (u2) -- (u3);
\node[font=\scriptsize, above] at ( 0.0,  0.7) {$c_{12}$};
\node[font=\scriptsize, left ] at (-0.7, -0.1) {$c_{13}$};
\node[font=\scriptsize, right] at ( 0.7, -0.1) {$c_{23}$};
\end{tikzpicture}
\caption{Triangle $K_3$ on $\{u_1, u_2, u_3\}$ in $G^{VN}$.}
\label{fig:33-vn}
\end{subfigure}
\caption[The Wang-Dolecek-Wesel $(3, 3)$-absorbing $\leftrightarrow$ triangle bijection]{The Wang-Dolecek-Wesel correspondence of Lemma~\ref{lem:33-triangle}, illustrated. \textbf{(a)} An elementary $(3, 3)$-absorbing set in a column-weight-$3$ girth-$\geq 6$ Tanner graph: three variables $u_1, u_2, u_3$ (blue), three degree-$2$ even checks $c_{12}, c_{13}, c_{23}$ (white squares) each shared by exactly two variables, and three degree-$1$ odd checks (red squares) each attached to exactly one variable. \textbf{(b)} The corresponding triangle in the variable-node graph $G^{VN}$: each pair-check $c_{ij}$ becomes the edge $u_i\!-\!u_j$, and the absorbing structure is reduced to a single triangle. The bijection runs both ways (Lemma~\ref{lem:33-triangle}), translating $(3, 3)$-absorbing-set enumeration into the counting of triangles realised by three distinct checks.}
\label{fig:33-bijection}
\end{figure}

\begin{proof}
The correspondence is established in \cite{wangdolecekwesel2012} via the cycle-consistency-matrix formalism for separable circulant-based codes; we give a graph-theoretic restatement valid for any Tanner graph satisfying the hypotheses.

$(\Rightarrow)$ Let $A = \{u_1, u_2, u_3\}$ be an elementary $(3, 3)$-absorbing set. The total incidence count from $A$ is $\sum_i \gamma = 9$. By elementarity each check in $N(A)$ has degree $1$ or $2$ in $G_A$; the $b = 3$ degree-$1$ checks contribute $3$ incidences, leaving $6$ incidences from $3$ degree-$2$ checks, so $|N(A)| = 6$. The strong-majority condition forces each $u_i$ to have at least $2$ even-degree-check neighbours in $G_A$, so (with $\gamma = 3$) exactly $2$ degree-$2$ and $1$ degree-$1$ check neighbour. Thus the $3$ degree-$2$ checks form a triangle on $\{u_1, u_2, u_3\}$ in $G^{VN}$, realised by three pairwise distinct checks.

$(\Leftarrow)$ Conversely, given a triangle $\{u_1, u_2, u_3\}$ in $G^{VN}$ realised by three pairwise distinct checks $c_{12}, c_{13}, c_{23}$ (equivalently, a simple $6$-cycle of $G$), each $u_i$ has its three checks split as the two pair-checks $c_{ij}, c_{ik}$ ($j, k \neq i$) and one extra check $c''_i$. The check $c''_i$ cannot also be incident to $u_j$ for any $j \neq i$, else $u_i$ and $u_j$ would share two distinct checks ($c_{ij}$ and $c''_i$), giving a $4$-cycle and contradicting $g(G) \geq 6$. Thus $c''_i$ has degree exactly~$1$ in $G_A$. The same argument shows $c''_i \neq c''_j$ for $i \neq j$. So $|N(A)| = 6$, with $3$ degree-$2$ and $3$ degree-$1$ checks; each $u_i$ has $2$ even-degree and $1$ odd-degree check neighbour in $G_A$, satisfying strong majority. So $A$ is an elementary $(3, 3)$-absorbing set.
\end{proof}

The variable-node graph admits a clean spectral description.

\begin{lem}\label{lem:VN-trace}
Let $G$ be a Tanner graph with biadjacency matrix $H$, constant variable-node degree $\gamma$, and $g(G) \geq 6$. Then
\[
    A_{VN} \;=\; H^T H \;-\; \gamma\, I,
\]
and consequently for every $k \geq 0$
\begin{equation}\label{eq:AVN-trace}
    \operatorname{tr}(A_{VN}^k) \;=\; \sum_{l=0}^{k} \binom{k}{l}\,(-\gamma)^{k-l}\,\operatorname{tr}\bigl((H^T H)^l\bigr).
\end{equation}
\end{lem}

\begin{proof}
The off-diagonal entry $(H^T H)_{u, u'}$ counts the common check-neighbours of $u, u'$, which is $0$ or $1$ under $g(G) \geq 6$ (else a $4$-cycle); the diagonal entry equals the variable-node degree $\gamma$. So $H^T H = A_{VN} + \gamma I$. Identity~\eqref{eq:AVN-trace} follows by binomial expansion of $A_{VN}^k = (H^T H - \gamma I)^k$ and linearity of trace.
\end{proof}

For QC-GHP codes the right-hand side of~\eqref{eq:AVN-trace} is fully spectral: $$\operatorname{tr}((H_X^T H_X)^l) = \operatorname{tr}((H_X H_X^T)^l) = \sum_{j=0}^{\ell-1}\operatorname{tr}(M_j^l)$$ for $l \geq 1$ by Theorem~\ref{prop:spectral-trace}, while $$\operatorname{tr}((H_X^T H_X)^0) = \operatorname{tr}(I_{(n_1+r_1)\ell}) = (n_1 + r_1)\ell.$$

Triangles in $G^{VN}$ recover the $6$-cycle count of $G$:

\begin{lem}\label{lem:VN-triangles}
For $G$ as above, the triangles of $G^{VN}$ are of two kinds: those realised by three distinct checks, which are in bijection with the simple $6$-cycles of $G$, and those realised by a single check of degree at least $3$. Consequently
\[
    \#K_3(G^{VN}) \;=\; N_6(G) \;+\; \sum_{c \in W}\binom{\deg(c)}{3},
    \]
 or equivalently
\[    
    N_6(G) \;=\; \tfrac{1}{6}\operatorname{tr}(A_{VN}^3) \;-\; \sum_{c \in W}\binom{\deg(c)}{3}.
\]
\end{lem}

\begin{proof}
Each edge of a triangle $\{u,v,w\}$ in $G^{VN}$ is realised by a check: $u,v$ share $c_{uv}$, and likewise $c_{vw}, c_{uw}$, each pair sharing exactly one check since $g(G)\geq 6$. If two of these coincide, say $c_{uv}=c_{vw}=c$, then $c$ is adjacent to all of $u,v,w$, so $c=c_{uw}$ as well; hence the three checks are either pairwise distinct or all equal.
If they are pairwise distinct, the closed walk $u\,c_{uv}\,v\,c_{vw}\,w\,c_{uw}\,u$ is a simple $6$-cycle (a repeated vertex would force a cycle shorter than $6$); conversely every simple $6$-cycle of the bipartite graph $G$ alternates variable/check and projects to such a triangle, so this is a bijection.
If the three checks are all equal to one check $c$, then $\deg(c)\geq 3$ and $\{u,v,w\}$ is one of its $\binom{\deg(c)}{3}$ neighbour-triples; no such triangle is a $6$-cycle, and none is counted for two different checks (two checks adjacent to the same three variables would create a $4$-cycle, impossible at $g(G)\geq 6$), giving $\sum_{c}\binom{\deg(c)}{3}$ single-check triangles.
Since $A_{VN}$ is the adjacency matrix of the simple graph $G^{VN}$, $\tfrac16\operatorname{tr}(A_{VN}^3)=\#K_3(G^{VN})$, and the two displayed identities follow.
\end{proof}

Combining Lemmas~\ref{lem:33-triangle}, \ref{lem:VN-trace} and~\ref{lem:VN-triangles} with Theorem~\ref{prop:spectral-trace} gives the main result of this subsection: in the QC-GHP setting, the count of $(3, 3)$ elementary absorbing sets, equivalently the $6$-cycle count $N_6$, admits an explicit closed-form Fourier expression.

\begin{prop}[$(3,3)$-absorbing-set spectral count for QC-GHP codes]\label{prop:33-spectral}
Let $G_X$ be the $X$-Tanner graph of a QC-GHP code over $\mathbb{Z}_\ell$ with $\gamma = 3$ and $g(G_X) \geq 6$. Then
\[
    \#\,(\text{elementary }(3, 3)\text{-absorbing sets}) \;=\; N_6(G_X),
\]
\noindent and
{\footnotesize{
\begin{equation}\label{eq:N6-spectral}
    N_6(G_X) \;=\; \frac{1}{6}\!\left[\,\sum_{j=0}^{\ell-1}\operatorname{tr}(M_j^3) \;-\; 3\gamma\sum_{j=0}^{\ell-1}\operatorname{tr}(M_j^2) \;+\; 3\gamma^2\sum_{j=0}^{\ell-1}\operatorname{tr}(M_j) \;-\; \gamma^3\,(n_1 + r_1)\,\ell\,\right] \;-\; \sum_{c \in W}\binom{\deg(c)}{3}.
\end{equation}
}}
\end{prop}

The bracket in Proposition \ref{prop:33-spectral} equals $\operatorname{tr}(A_{VN}^3)$, so the first term equals $\tfrac16\operatorname{tr}(A_{VN}^3)=\#K_3(G_X^{VN})$; the correction $\sum_{c}\binom{\deg(c)}{3}$ removes the triangles of $G_X^{VN}$ induced by a single check of degree $\geq 3$ (Lemma~\ref{lem:VN-triangles}). For a $(\gamma,k)$-biregular QC-GHP code with $m=r_1\ell$ checks of degree $k$ it equals $r_1\ell\binom{k}{3}$; for the $[\![882,24]\!]$ code this is $441\binom{6}{3}=8820$.
For generalised bicycle codes ($r_1 = n_1 = 1$), each $M_j$ collapses to the scalar $|a(\omega^j)|^2 + |b(\omega^j)|^2$, and~\eqref{eq:N6-spectral} reduces to a polynomial in those scalars together with the constant $\sum_{c}\binom{\deg(c)}{3}=\ell\binom{w_a+w_b}{3}$, computable in $O(\ell \log \ell)$ time via FFT.

The combinatorial half, $$\#(\text{elementary (3, 3)-absorbing sets}) = N_6(G_X),$$ is the Wang-Dolecek-Wesel observation~\cite{wangdolecekwesel2012} in $G^{VN}$ language. The spectral half, the closed-form expression~\eqref{eq:N6-spectral} for $N_6(G_X)$ in terms of the QC-GHP Fourier evaluations $A(\omega^j)$ and $b(\omega^j)$, is the contribution of Section~\ref{sec:spectral-cycles}; the directed-edge-matrix DFT method of Karimi and Banihashemi~\cite{karimi2012} gives an equivalent count for classical QC-LDPC codes, but the explicit two-block QC-GHP form via the $r_1 \times r_1$ matrices $M_j$ does not appear there.

\begin{rem}[Other absorbing-set sizes]
For column-weight-$3$ girth-$6$ codes the smallest elementary absorbing-set types are $(3, 3)$ and $(4, 2)$, see~\cite{wangdolecekwesel2012}. The $G^{VN}$ dictionary handles $(4, 2)$ in principle: an elementary $(4, 2)$-absorbing set corresponds to an induced ``$K_4$ minus an edge'' on four variables in $G^{VN}$, with a degree-$1$ check attached to each of the two vertices of $G^{VN}$-degree~$2$ within the configuration. Counting induced ``$K_4$ minus an edge'' subgraphs requires $\operatorname{tr}(A_{VN}^4)$ together with subgraph-elimination corrections, themselves expressible via the spectral data but tied to the additive $T_3$ machinery of Remark~\ref{rem:6cycles}; we defer this computation to future work. The case $(4, 0)$ corresponds to induced $K_4$ in $G^{VN}$ but is generally rare or absent in the regime relevant to QC-GHP codes (where girth at most $6$ is forced by the CSS commutation constraint \cite{AmirzadePanarioSadeghi2024}), so we do not develop it further.
\end{rem}

\begin{rem}[Comparison with adjacent classical QC-LDPC machinery]\label{rem:higher-Nk-pushes}
This remark compares the spectral-trace identity of Theorem~\ref{prop:spectral-trace} with two adjacent classical pieces of QC-LDPC machinery, both of which yield cycle counts in the QC-GHP setting via slightly different routes: the directed-edge-matrix framework of Karimi-Banihashemi~\cite{karimi2012}, and the moment-recursion of Smarandache-Mitchell~\cite{SM2026}. The two approaches give equivalent answers up to a change of basis, but differ in what they require as input and what their output looks like.
The directed-edge-matrix framework of Karimi and Banihashemi~\cite{karimi2012} gives, for any graph of girth $g$, the clean cycle count $N_k = \operatorname{tr}(A_e^k)/(2k)$ for $k < 2g$, where $A_e$ is the Hashimoto $2|E| \times 2|E|$ directed-edge matrix; no trace-correction terms (cf.\ Lemma~\ref{lem:VN-trace}) are needed. For a QC protograph code over $\mathbb{Z}_\ell$, the matrix $A_e^2$ has block-circulant structure with blocks of size $|E_b| \times |E_b|$ ($|E_b|$ = number of edges of the base graph), and the same DFT block-diagonalisation as Theorem~\ref{prop:spectral-trace} reduces $N_k$ to a sum of traces of $\ell$ small matrices. The construction applies verbatim to the $X$-Tanner graph $G_X$ of any QC-GHP or generalised bicycle code over $\mathbb{Z}_\ell$, yielding closed-form expressions for $N_4, N_6, \ldots, N_{2g-2}$ via Fourier evaluations of base-graph data, on a larger ambient matrix ($2|E|$ vs.\ $r_1 \ell$) but without the lower-cycle correction terms that appear in our adjacency-based identity. We do not develop this route further.

Complementarily, the Smarandache-Mitchell recursion~\cite[Cor.~4]{SM2026} expresses $N_{2j}$ for $2 \leq j \leq 7$ in any $(\gamma, k)$-biregular Tanner graph as a polynomial in the moments $\mu_j \coloneqq \operatorname{Tr}((HH^T)^j)$ and the lower cycle counts $N_4, N_6, \ldots, N_{2j-2}$, under a girth hypothesis that grows with $j$: girth $\geq 4$ for $N_4, N_6$, girth $\geq 6$ for $N_8, N_{10}$, and girth $\geq 8$ for $N_{12}, N_{14}$. By Theorem~\ref{prop:spectral-trace}, $\mu_j = \sum_{j'=0}^{\ell-1}\operatorname{tr}(M_{j'}^j)$ in the QC-GHP setting; combined with the recursion this gives a Fourier expression for $N_{2j}$ in terms of the same Fourier matrices. In particular, the girth-$6$ hypothesis under which the rest of this section operates already suffices for $N_8$ and $N_{10}$, so these cycle counts of the running $[\![882, 24]\!]$ QC-GHP code (girth $6$) are computable in closed form from its Fourier matrices; we do not work out the expressions here. Translating $N_{2j}(G_X)$ into an absorbing-set count beyond the $(3, 3)$ bijection of Proposition~\ref{prop:33-spectral} would require an analogue of Lemma~\ref{lem:33-triangle} for length-$2j$ cycles in $G_X$ (equivalently, $j$-cycles in $G^{VN}$ on $j$ distinct variables under a suitable girth hypothesis); we do not pursue this here.
\end{rem}

\section{Spectral lower bound on the QC-GHP stopping distance}\label{sec:qcghp-stopping-spectral}
Combining the bipartite version of the expander Mixing Lemma \aida{maybe we should add the statement in the spectral preliminaries?} (Haemers~\cite{H1995}; see also \cite[\S 4.3, 4.8]{brouwer2011spectra}) with the Fourier matrices of Theorem~\ref{prop:spectral-trace} yields a spectral lower bound on the QC-GHP stopping distance.

\begin{prop}[Spectral gap bound on stopping-set size in QC-GHP codes]\label{prop:stopping-spectral}
Let $G_X$ be the $X$-Tanner graph of a $(\gamma, k)$-biregular QC-GHP code over $\mathbb{Z}_\ell$, with $m = r_1 \ell$ check nodes and Fourier matrices $\{M_j\}_{j=0}^{\ell - 1}$ from Theorem~\ref{prop:spectral-trace}. The second-largest singular value $\sigma_2(H_X)$ of $H_X$ satisfies
\begin{equation}\label{eq:sigma2-Mj}
    \sigma_2(H_X)^2 \;=\; \max\bigl(\lambda_2(M_0),\,\max_{j \neq 0}\lambda_{\max}(M_j)\bigr),
\end{equation}
where $\lambda_2(M_0)$ is the second-largest eigenvalue of $M_0$. Every non-empty stopping set $S \subseteq V_L$ of size $a$ (Definition~\ref{def:stopping}) satisfies
\begin{equation}\label{eq:stopping-spectral}
    a\!\left(m - \frac{a\gamma}{2}\right) \;\leq\; \left(\frac{\sigma_2(H_X)\, m}{\gamma}\right)^{\!2}.
\end{equation}
Writing $C = (\sigma_2(H_X)\,m/\gamma)^2$ and setting $a_\pm = \bigl(m \pm \sqrt{m^2 - 2\gamma C}\bigr)/\gamma$, inequality~\eqref{eq:stopping-spectral} is informative precisely when $\sigma_2(H_X)^2 < \gamma/2$; in this regime every non-empty stopping set $S$ satisfies $|S| \leq a_-$ or $|S| \geq a_+$, so the stopping distance is bounded below by $a_+$ whenever the small-size range $(0, a_-]$ is excluded by separate (e.g.\ girth- or degree-based) considerations.
\end{prop}

\begin{proof}
See Figure~\ref{fig:expander-stopping}. By Definition~\ref{def:stopping}, a stopping set $S$ satisfies $|N(S)| \leq a\gamma/2$, since every check in $N(S)$ has at least two neighbours in $S$. For a $(\gamma, k)$-biregular code with $m$ checks and $n$ variables (so $|E| = \gamma n = k m$), the expander mixing lemma applied to the bipartite Tanner graph gives, for any $S \subseteq V_L$ and $T \subseteq V_R$,
\[
    \bigl|\,|E(S, T)| - \tfrac{|E|}{mn}|S||T|\,\bigr| \;\leq\; \sigma_2(H_X) \sqrt{|S||T|}.
\]
Taking $T = V_R \setminus N(S)$ makes $E(S, T) = \emptyset$, so the EMI reduces to $\frac{\gamma a |T|}{m} \leq \sigma_2(H_X)\sqrt{a |T|}$. Squaring and dividing by $a|T| > 0$ yields
\[
    a \, |T| \;\leq\; \left(\frac{\sigma_2(H_X)\, m}{\gamma}\right)^{\!2}.
\]
Substituting $|T| = m - |N(S)| \geq m - a\gamma/2$ gives \eqref{eq:stopping-spectral}. For~\eqref{eq:sigma2-Mj}, the DFT block-diagonalisation of Theorem~\ref{prop:spectral-trace} identifies $\operatorname{spec}(H_X H_X^T)$ with the disjoint union of $\operatorname{spec}(M_j)$ over $j = 0, \ldots, \ell - 1$. The Perron eigenvalue $\sigma_1(H_X)^2 = \gamma k = \lambda_{\max}(M_0)$, and the next-largest eigenvalue in this union is the larger of $\lambda_2(M_0)$ and $\max_{j \neq 0}\lambda_{\max}(M_j)$.
\end{proof}

\begin{rem}[Regime of applicability of Proposition \ref{prop:stopping-spectral}]\label{rem:regime}
The non-vacuity condition $\sigma_2(H_X)^2 < \gamma/2$ is restrictive. For $(\gamma, k)$-biregular bipartite graphs the Alon-Boppana-type lower bound~\cite{fengli1996, lisole1996} forces $\sigma_2(H_X) \geq \sqrt{\gamma - 1} + \sqrt{k - 1} - o(1)$ as $\ell \to \infty$, so for the column-weight-$3$ family of interest (including the $[\![882, 24]\!]$ code of Example~\ref{ex:ghp}, with $\gamma = 3$, $k = 6$) the inequality is structurally vacuous: $\sigma_2^2 \gtrsim (\sqrt{2} + \sqrt{5})^2 \approx 13$ while $\gamma/2 = 1.5$. Proposition~\ref{prop:stopping-spectral} therefore reads as a \emph{framework-level} statement: it shows how the QC-GHP Fourier matrices control the stopping distance \emph{whenever the second singular value can be driven below $\sqrt{\gamma/2}$}, a regime accessible to designs with substantially larger column weight or to Ramanujan-type lifted-product constructions, rather than to the small-$\gamma$ examples of Section~\ref{sec:background}. Closing this gap, either by a Tanner-style multi-eigenvalue tightening or by replacing the EMI step with a more parameter-aware spectral inequality, remains an open direction.
\end{rem}


\begin{figure}[htp]
\centering
\begin{tikzpicture}[thick,
  vIn/.style={circle, draw=black, fill=ugentblue, minimum size=8pt, inner sep=0pt},
  vOut/.style={circle, draw=black, fill=white, minimum size=8pt, inner sep=0pt},
  cN/.style={regular polygon, regular polygon sides=4, draw=black, fill=red!25, minimum size=10pt, inner sep=0pt, very thick},
  cT/.style={regular polygon, regular polygon sides=4, draw=black, fill=white, minimum size=10pt, inner sep=0pt}]
\foreach \k in {1,...,4} {
  \node[vIn] (s\k) at (-2.4, 1.6 - \k*0.6) {};
}
\foreach \k in {5,...,8} {
  \node[vOut] (s\k) at (-2.4, 1.6 - \k*0.6) {};
}
\node[font=\scriptsize, left=2pt] at (-2.4, 1.0) {$S$};
\node[font=\scriptsize, left=2pt] at (-2.4, -1.5) {$V_L \setminus S$};
\foreach \k in {1,2,3} {
  \node[cN] (n\k) at ( 2.4, 1.3 - \k*0.6) {};
}
\foreach \k in {4,...,8} {
  \node[cT] (n\k) at ( 2.4, 1.3 - \k*0.6) {};
}
\node[font=\scriptsize, right=2pt] at ( 2.4, 0.7) {$N(S)$};
\node[font=\scriptsize, right=2pt] at ( 2.4, -1.9) {$T = V_R \setminus N(S)$};
\draw[ugentblue!70] (s1) -- (n1);  \draw[ugentblue!70] (s2) -- (n1);
\draw[ugentblue!70] (s2) -- (n2);  \draw[ugentblue!70] (s3) -- (n2);
\draw[ugentblue!70] (s3) -- (n3);  \draw[ugentblue!70] (s4) -- (n3);
\draw[ugentblue!70] (s1) -- (n3);
\node[red, font=\scriptsize, align=center] at (0, -2.2) {no edges between $S$ and $T$};
\node[align=center, font=\scriptsize] at (0, 1.4) {$|N(S)| \leq a\gamma/2$\\ (each red check sees $\geq 2$ neighbours in $S$)};
\end{tikzpicture}
\caption[The expander mixing argument for stopping sets]{The expander mixing lemma proof of Proposition~\ref{prop:stopping-spectral}. A stopping set $S \subseteq V_L$ (blue, size $a$) has $|N(S)| \leq a\gamma/2$ because every check in $N(S)$ (red squares) sees at least two neighbours in $S$, and the total $a\gamma$ incidences from $S$ are spread over the checks of $N(S)$. The expander mixing lemma is then applied to the empty edge set between $S$ and the test set $T = V_R \setminus N(S)$ of non-neighbour checks: the deviation $|E(S, T)| - |E||S||T|/(mn)$ from the expected count is bounded by $\sigma_2(H_X)\sqrt{|S||T|}$, and since the actual count is $0$ this forces $a$ to satisfy~\eqref{eq:stopping-spectral}.}
\label{fig:expander-stopping}
\end{figure}
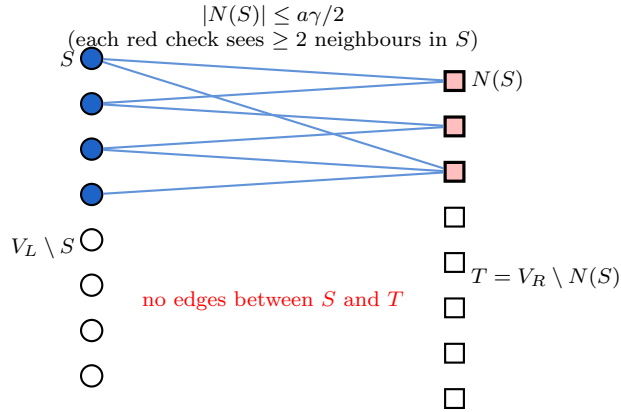

\section{Concluding remarks}\label{sec:conclusion}
We have developed a quantum-specific spectral framework for the lifted-product (QC-GHP) and generalised bicycle Tanner graphs of the leading quantum LDPC constructions. The central tool is the DFT block-diagonalisation of $H_X H_X^T$ (Theorem~\ref{prop:spectral-trace}), which exploits the two-block tensor form $H_X = [\,\widetilde{H}_1 \mid I_{r_1}\otimes\widetilde{B}^T\,]$ by absorbing the second block into a scalar shift $|b(\omega^j)|^2 I_{r_1}$ and reduces an $(r_1\ell)\times(r_1\ell)$ moment-trace computation to $\ell$ traces of $r_1\times r_1$ Hermitian matrices, extending the classical single-block reductions of Smarandache--Flanagan~\cite{SF2009} and Karimi--Banihashemi~\cite{karimi2012} to the quantum setting. To illustrate applications of the proposed spectral method, we obtain, for generalised bicycle codes, a $4$-cycle count through additive energies (Proposition~\ref{prop:4cycles}) and a joint Sidon characterisation of girth~$6$ echoing Fossorier's criterion~\cite{fossorier2004} (Corollary~\ref{cor:sidon}); via the Wang--Dolecek--Wesel triangle bijection~\cite{wangdolecekwesel2012}, a closed-form Fourier count of the $(3,3)$ elementary absorbing sets of column-weight-$3$ QC-GHP codes (Proposition~\ref{prop:33-spectral}), addressing a question raised by the framework of Morris, Pllaha, and Kelley~\cite{morris2024}; and, through the bipartite expander mixing lemma applied to the second-largest singular value computed from the same Fourier matrices, a spectral-gap lower bound on stopping-set sizes (Proposition~\ref{prop:stopping-spectral}).

We conclude with a few open problems.

\begin{itemize}
    \item \emph{Higher even-cycle and absorbing-set counts.} Extending the closed-form counts beyond $4$- and $6$-cycles requires the additive $3$-energy ($T_3$) computations and the girth-$8$ criterion noted in Remark~\ref{rem:6cycles}, together with a length-$2j$ analogue of the $(3,3)$-to-triangle bijection in order to translate higher cycle counts into absorbing-set counts.
    \item \emph{Non-vacuous stopping-set bounds.} For the column-weight-$3$ examples studied here the spectral-gap bound is vacuous, since the second singular value exceeds $\sqrt{\gamma/2}$. Identifying QC-GHP design parameters, for instance larger column weight or \linebreak Ramanujan-type lifted-product constructions, under which the bound becomes informative is left for future work, as is a Tanner-style multi-eigenvalue tightening of the expander mixing step.
    \item \emph{Other elementary absorbing-set types.} The $(4,2)$ and $(4,0)$ absorbing-set cases are accessible in principle through the same spectral data, via induced ``$K_4$ minus an edge'' and induced $K_4$ counts in the variable-node graph, but they require the additional subgraph-elimination corrections sketched in Remark~\ref{rem:6cycles}.
\end{itemize}

\paragraph{Acknowledgements.}
Aida Abiad is supported by NWO (Dutch Research Council) through the grant VI.Vidi.213.085. We thank Kirsten Morris for bringing the quantum LDPC construction of Panteleev-Kalachev to our attention.

\printbibliography[heading=bibintoc]

\end{document}